\documentclass{amsart}
\usepackage{amssymb,amsmath, amsthm,latexsym}
\usepackage{graphics}
\usepackage{psfrag}
\usepackage{amscd}
\usepackage{graphicx}
\newcommand{\cal}[1]{\mathcal{#1}}
\theoremstyle{plain}
\newtheorem{theo}{Theorem}

\newtheorem{lemma}{Lemma}[section]
\newtheorem{theorem}[lemma]{Theorem}
\newtheorem{proposition}[lemma]{Proposition}
\newtheorem{corollary}[lemma]{Corollary}
 \theoremstyle{definition}

\newtheorem{definition}[lemma]{Definition}
\newtheorem{remark}[lemma]{Remark}
\let\egthree=\phi
\let\phi=\varphi
\let\varphi=\egthree

\begin{document}
\title{Incompressible surfaces in rank one locally symmetric spaces}
\author{Ursula Hamenst\"adt}
\thanks{Partially supported by 
ERC Grant 10160104\\
AMS subject classification:53C35,53B21,37A25}
\date{November 28, 2014}


\begin{abstract} We show that cocompact lattices in rank one simple Lie
groups of non-compact type distinct from $SO(2m,1)$ $(m\geq 1)$ 
contain surface subgroups.
\end{abstract}

\maketitle

\section{Introduction}

In recent seminal work, Jeremy Kahn and Vladimir Markovic 
showed that every cocompact lattice in $SO(3,1)$ contains
infinitely many surface subgroups, i.e. subgroups which are
isomorphic to the fundamental group of  a closed surface
\cite{KM12}.
Moreover, as a subgroup of $SO(3,1)$, each of these
groups is quasi-Fuchsian, and every round circle in 
the ideal boundary of hyperbolic three-space
is the limit in the Hausdorff topology of a sequence of 
limit sets of such groups.

In view of a conjecture of Gromov that every one-ended
hyperbolic group contains a surface subgroup, it seems
desirable to extend the result of Kahn and Markovic
to a larger class of groups. The goal of this paper is to
undertake such an extension to cocompact lattices in 
simple Lie groups of rank one which are distinct from 
$SO(2m,1)$ for $m\geq 1$.

\begin{theo}\label{mainresult}
Let $G$ 
be a simple rank one Lie group of non-compact type 
distinct from $SO(2m,1)$ for some $m\geq 1$ 
and let
$\Gamma<G$ be a cocompact lattice. Then $\Gamma$ contains
surface subgroups.
\end{theo}

Since by Selberg's lemma 
every finitely generated subgroup of a linear group 
contains
a torsion-free subgroup of finite index
(see \cite{Ra94} for a proof), 
the theorem is an immediate
consequence of the statement that 
every closed locally symmetric
manifold $M$ of negative curvature 
which is different from a hyperbolic manifold of even dimension
contains closed incompressible immersed
surfaces.

The proof of this fact 
uses the strategy developed by Kahn and
Markovic. Namely, the surfaces will be glued from
immersed incompressible pants with geodesic boundary.
These pants are viewed as topological objects, but they
have geometric realizations as piecewise ruled surfaces which 
are close to totally geodesic pants.
Closeness in this sense can be quantified, and 
this allows to establish a glueing condition for such
pants which results in an incompressible surface.

This differential geometric viewpoint is the main novelty
of this work. It also results in a significant 
simplification of the original construction 
for closed hyperbolic three-manifolds. 

The glueing condition can be expressed as a system
of linear glueing equations, and the task is then
to find a non-negative solution. For manifolds of higher
dimension, this turns out to be much more difficult than
for hyperbolic three-manifolds, and this is also
the place where our argument is not valid for 
even dimensional hyperbolic manifolds. 
We refer to Section 7 for a more detailed discussion of this
difficulty.

We do not formulate 
an approximation result along the lines
of the work of Kahn and Markovic since we are not aware of any
interesting application, although 
suitable versions of such an
approximation result hold as well.
%

The paper is completely self-contained. 
Some of our arguments are valid
for arbitrary closed negatively curved manifolds 
with generic metrics, but
we do not know whether the beautiful 
result of Kahn and Markovic holds true in this setting.

In Section \ref{cannon} we collect some background and
tools used later on, and we give a more
detailed overview of the argument.

%

\bigskip
{\bf Acknowledgement:} The major part of 
this work was carried out while I visited
the IPAM in Los Angeles. I am grateful to IPAM for the
hospitality, and to Vladimir Markovic for
useful discussions. I am also grateful for the anonymous
referees for pointing out an erraneous statement in an earlier
version of this work, for informing me about the
papers \cite{B12,S12} and for valuable
comments geared at improving the exposition.

\section{Setup and tools}\label{cannon}

This section has three parts. 
In the first part we give a short outline of the proof,
enhancing its similiarities and differences to the 
argument of Kahn and Markovic. 
The second part summarizes 
those properties of rank one symmetric spaces
which are used later on. In the third part we
give a simple (and well known to the experts but hard to find
 in the literature) 
criterion for incompressibility of a surface in 
a closed nonpositively curved manifold.

\subsection{Structure of the proof and outline of the paper}

To prove Theorem \ref{mainresult} from the introduction, we
construct closed surfaces $S$ with 
piecewise smooth locally ${\rm CAT}(-1/2)$-metrics and 
piecewise smooth isometric immersions 
$f:S\to M$. The surfaces are constructed in such a way that
they admit a distinguished 
\emph{pants decomposition} ${\cal P}$ 
which is mapped by $f$ to a collection of closed geodesics in $M$.
Such a pants decomposition consists of a collection of 
pairwise disjoint simple closed geodesics which decompose
$S$ into \emph{pairs of pants}, i.e. 
three-holed spheres. The image under $f$ of 
each pair of pants is a 
piecewise smooth isometrically 
immersed surface in $M$ which is geometrically 
close to a totally geodesic totally real hyperbolic subsurface.

To show that these surfaces are incompressible we associate
to each component of the thin part of such a pair of pants 
an immersed totally geodesic totally real hyperbolic plane in $M$.
We then lift these immersed hyperbolic
planes to the universal covering $\tilde M$ of $M$ 
and use these chains of planes to establish a sufficient
condition for incompressibility which relies on the simple
characterization formulated in Proposition \ref{cannonmap}.
The construction is carried out in Sections 5 and 6, and 
it is valid for symmetric spaces of higher rank as well, 
where we work with totally geodesic hyperbolic
planes of constant curvature $-1$ (which therefore intersect
a given maximal flat in at most one geodesic).

The construction of the basic building blocks
of the surfaces, the pairs of pants, is carried out in 
Section 4 and uses 
the strategy from \cite{KM12}. 
In our approach, it is essential to equip these pairs of 
pants with a geometric structure and to 
keep track of the
monodromy defined by parallel transport of frames 
along boundary geodesics.

Glueing the pants to closed surfaces which satisfy the
condition for incompressibility established in Section 6 
amounts to solving a glueing equation. For this we use
ideas from \cite{KM12}: project the Lebesgue measure
on suitable frame bundles to the sphere bundles over
closed geodesics and use these projected
measures to attach pants along their boundary geodesics
in such a way that the glueing equation is fulfilled.

The projected measures are not so easy to control. 
If the monodromy of a closed geodesic 
has a single fixed real line (which may happen in 
even dimensional hyperbolic manifolds) 
then it is not clear whether
the projected measure is 
symmetric enough to construct 
a solution of the glueing equation, and this difficulty 
is the reason why our proof of Theorem \ref{mainresult} fails
for even dimensional real hyperbolic manifolds. 
We refer to the short section \ref{even} for a
more detailed discussion of this difficulty.


Section 3 contains some distance estimates in  
universal coverings of 
surfaces with ${\rm CAT}(-1)$-metric which are equipped with 
a pants decomposition with specific properties. 
The surfaces we construct will be equipped with pants
decompositions with precisely these properties. 


The results in this work are formulated so that
they can easily be used in more general situations than the one
we consider here. 

For a nice exposition of the work of Kahn and Markovic we 
refer the reader to \cite{B12}. An alternative approach
to incompressibility of immersed 
surfaces in hyperbolic 3-manifolds can be found
in \cite{S12}.

\subsection{The geometry of rank one symmetric spaces}

A rank-one symmetric space is a simply connected
Riemannian manifold $\tilde M$ of negative sectional curvature
whose group of orientation preserving isometries
is one of the simple Lie groups
$G=SO(n,1),SU(n,1),Sp(n,1),F_4^{-20}$.
More precisely, $\tilde M=G/K$ where $K<G$ is a maximal
compact subgroup and the Riemannian metric on 
$\tilde M$ is induced from the Killing form. We always assume that
the maximum of the sectional curvature of 
$\tilde M$ equals $-1$.
The geometry of these symmetric spaces can
be described as follows. 

The symmetric space $SO(n,1)/SO(n)$ is the 
\emph{hyperbolic
$n$-space} ${\bf H}^n$ of constant curvature $-1$. For any given point 
$p\in {\bf H}^n$, the stabilizer of $p$ in $SO(n,1)$ acts transitively
on the unit sphere $T_p^1{\bf H}^n$ in the tangent space of 
${\bf H}^n$ at $p$. This action is just the standard action of 
$SO(n)$ on $S^{n-1}$. 
The stabilizer in $SO(n,1)$ of a unit
tangent vector $X\in T^1{\bf H}^n$ equals the subgroup 
$SO(n-1)$ of $SO(n)$. It acts simply transitively on the
space of oriented 
orthonormal frames in the orthogonal complement of $X$.
For each $k\geq 1$, the group $SO(n)$ also acts transitively
on the Grassmannian of $k$-dimensional linear subspaces of 
$T_p{\bf H}^n$. Each such subspace is the tangent space at $p$
of a unique totally geodesic submanifold of ${\bf H}^n$ which is isometric
to a hyperbolic space of dimension $k$
(for $k=1$, this is just a geodesic). 

The symmetric space $SU(n,1)/S(U(n)U(1))$ is 
\emph{complex hyperbolic $n$-space}
$\mathbb{C}{\bf H}^n$ which is Hermitean
symmetric. Denote by $J$ the complex structure, viewed as
an automorphism of the tangent bundle of $\mathbb{C}{\bf H}^n$. 
For any given point $p\in \mathbb{C}{\bf H}^n$, 
the stabilizer of $p$ in $SU(n,1)$ acts transitively on the unit
sphere $T_p^1\mathbb{C}{\bf H}^n$ in the tangent space of 
$\mathbb{C}{\bf H}^n$ at $p$. This action is just the standard action
of $S(U(n)U(1))$ 
on the unit sphere in $\mathbb{C}^n$. The stabilizer in 
$S(U(n)U(1))$ of a unit tangent vector $X\in T^1\mathbb{C}{\bf H}^n$ 
equals the group $SU(n-1)$ which acts 
simply transitively 
on the space of unitary frames in the orthogonal complement of 
the span of $X$ and $JX$. 

Each real line in the tangent
space at $p$ is tangent to a unique complex one-dimensional
submanifold of $\mathbb{C}{\bf H}^n$ which is 
a totally geodesic embedded hyperbolic plane of constant curvature
$-4$. A tangent plane spanned by a vector $0\not=X$ and a vector $Y$ which
is orthogonal to the span of $X$ and $JX$ is \emph{totally real}
and tangent to a unique totally geodesic embedded hyperbolic
plane of constant curvature $-1$. The set of all totally real
planes containing the fixed vector $X$ can naturally be 
identified with a sphere of 
dimension $2n-3$. Parallel transport along a geodesic $\gamma$ 
commutes with the complex structure and hence it preserves the
sphere bundle over $\gamma$ corresponding to unit tangents 
which span together with $\gamma^\prime$ a
totally real plane. 
We refer to the monograph \cite{Go99} for more
information on complex hyperbolic space.

The symmetric space $Sp(n,1)/Sp(n)Sp(1)$ is 
\emph{quaternionic hyperbolic $n$-space} $\mathbb{H}{\bf H}^n$.
There is a two-sphere of complex structures on 
$\mathbb{H}{\bf H}^n$. Each real line in the 
tangent space at any point $p$ is contained in a unique
quaternionic line, and this line is tangent to a unique
totally geodesic embedded real hyperbolic $4$-space of 
constant curvature $-4$. 
A two-plane spanned by a vector $X$ and
a vector which is orthogonal to the quaternionic line determined 
by $X$ is tangent to a unique totally geodesic 
hyperbolic plane of constant curvature $-1$. As before, we call such a
plane \emph{totally real}. The set of all totally real planes through
$X$ can naturally be identified with a sphere of dimension $4n-5$.
The stabilizer of $X$ in $Sp(n,1)$ is the group $Sp(n-1)$. 
It acts simply transitively on the space
of orthonormal $\mathbb{K}$-frames in the orthogonal complement of
$X$. 
Parallel transport along a geodesic $\gamma$ commutes with 
the quaternionic structure and hence it preserves the sphere
bundle over $\gamma$ corresponding to unit vectors which 
span together with $\gamma^\prime$ a totally real plane.

The \emph{Cayley plane} 
${\rm Ca}{\bf H}^2=F_4^{-20}/{\rm Spin}(9)$ has a similiar 
geometric description. 
The isotropy representation of the group ${\rm Spin}(9)$ is the 
spin representation. The action of ${\rm Spin}(9)$ on the
unit sphere 
$T_p^1{\rm Ca}{\bf H}^2=S^{15}$ in the tangent space at the point $p$ is 
transitive. The stabilizer in ${\rm Spin}(9)$ 
of a unit vector $X$ is 
the subgroup ${\rm Spin}(7)$ which acts transitively on the
unit sphere $S^7$ in the orthogonal complement of
the span of $X$ over the Cayley numbers.   
The unit tangent vector $X$ determines a unique totally geodesic embedded
hyperbolic $8$-space of constant curvature $-4$. 
A unit tangent vector orthogonal to the tangent space of this hyperbolic
space spans with $X$ the tangent space 
of a 
totally real hyperbolic plane of constant curvature 
$-1$. The monograph \cite{W11} (see in particular Section 8.2)
contains all information summarized above, and we refer to it for
a more detailed discussion.

Write $\tilde M=G/K$ and
let $T^1\tilde M$ be the unit tangent bundle of $\tilde M$.
For a unit tangent vector $v\in T^1\tilde M$ with foot-point $p$  
let 
\[v_{\mathbb{K}}^\perp\] be the 
$\mathbb{K}$-orthogonal complement of $v$ in 
$T_p\tilde M$. Here we put $\mathbb{K}=\mathbb{R}$ if 
$G=SO(n,1)$, $\mathbb{K}=\mathbb{C}$ if $G=SU(n,1)$,
$\mathbb{K}=\mathbb{H}$ if $G=Sp(n,1)$, and $\mathbb{K}=
\mathbb{O}$ if $G=F_4^{-20}$.
A unit tangent vector $w\in T^1_p\tilde M$ orthogonal to $v$ is contained in  
$v_{\mathbb{K}}^\perp$ if and only if 
the plane in $T_p\tilde M$ spanned by $v$ and $w$ is \emph{real}, i.e. its
curvature
equals $-1$.


The orientation of $\tilde M$ induces an orientation on 
$v_{\mathbb{K}}^\perp$. In the case $G=SO(n,1)$ this orientation
is determined by the requirement that for every
positive basis $X_2,\dots,X_n$ of $v_{\mathbb{K}}^\perp$ the
basis $v,X_2,\dots,X_n$ of $T_pM$ is positive.
If $G=SU(n,1)$ (or $G=Sp(n,1),F_4^{-20}$) 
we use the fact that the complex structure
(or the quaternionic structure or the Cayley structure)  
defines an orientation on the $\mathbb{K}$-lines
in $T\tilde M$.

Define the \emph{$\mathbb{K}$-frame bundle} 
\[P:{\cal F}\to T^1\tilde M\] 
over $T^1\tilde M$ to be the bundle whose
fibre over a unit tangent vector $v\in T^1\tilde M$ 
is the space of all positive orthonormal $\mathbb{K}$-frames in 
$v_{\mathbb{K}}^\perp$.
If $G=SO(n,1)$ then ${\cal F}$ 
is an $SO(n-1)$-principal  
bundle over $T^1M$. 
In the case $G=SU(n,1)$ it is 
an $SU(n-1)$-principal bundle, if 
$G=Sp(n,1)$ then it is an $Sp(n-1)$-principal bundle,
and in the case $G=F_4^{-20}$ it is a 
${\rm Spin}(7)$-principal bundle.

We need some information on parallel transport of 
$\mathbb{K}$-orthonormal frames 
along loops in totally geodesic hyperbolic 
planes $H\subset \tilde M$ of curvature $-1$, i.e.
totally real planes.
Recall that parallel transport along a loop 
$\gamma:[0,1]\to \tilde M$ is a $\mathbb{K}$-linear
isometry of $T_{\gamma(0)}\tilde M$.

For a linear subspace  $L\subset T\tilde M$ 
denote by $L\otimes \mathbb{K}$ the 
$\mathbb{K}$-linear subspace of $T\tilde M$ spanned by $L$.

\begin{lemma}\label{paralleltransport}
Let $\gamma:[0,1]\to H\subset \tilde M$ be a smooth
loop. Parallel transport along $\gamma$ induces the
identity on the orthogonal complement of 
$T_{\gamma(0)}H\otimes \mathbb{K}$.
\end{lemma}
\begin{proof} The real plane $H$ is totally geodesic in $\tilde M$
and therefore the tangent plane
of $H$ is invariant under parallel transport along
loops $\gamma$ in $H$.

Assume for the moment that $G=SO(n,1)$. Then for every 
unit vector $z\in T^1_{\gamma(0)}\tilde M$ which is 
orthogonal to $T_{\gamma(0)}H$,  there is a unique
totally geodesic hyperbolic 3-space $Q\supset H$ embedded in 
$\tilde M$ whose tangent space at $\gamma(0)$ equals the span of 
$T_{\gamma(0)}H$ and $z$. Since $Q$ is totally geodesic, parallel transport
along smooth curves in $Q$ preserves the tangent space of $Q$.
Moreover, it preserves the orientation of a frame.
Thus parallel transport along $\gamma$  
maps $z$ to itself. This shows the claim in the case
that $G=SO(n,1)$.

If $G=SU(n,1)$ then observe that the complex structure $J$ is 
invariant under parallel transport.
On the other hand, the above reasoning shows that 
the restriction of parallel transport along $\gamma\subset H$ 
to the orthogonal complement of $TH\otimes \mathbb{C}$ is
the identity. If $G=Sp(n,1)$ then the same argument applies to the
quaternionic span $T_{\gamma(0)}H\otimes \mathbb{H}$ 
of $T_{\gamma(0)}H$. The statement is empty for 
$G=F_4^{-20}$.
\end{proof}

The \emph{geodesic flow} $\Phi^t$ acts on $T^1\tilde M$. 
The \emph{frame flow} $\Psi^t$ is the lift of $\Phi^t$ 
to a flow on the bundle ${\cal F}$ which is defined by 
$\Psi^t(v,F)=(\Phi^tv,\Vert F) $. 
Here $\Vert F$ denotes parallel transport of the frame $F$  
along the projection
of the flow line $t\to \Phi^tv$ of the geodesic flow  
to a geodesic in $\tilde M$. The frame flow is well defined
since the Riemannian metric and the complex (or quaternionic)
structure are parallel.

Let $\Gamma<G$ be any torsion free cocompact lattice.
The existence of such a lattice was established by 
Borel \cite{Bo63}.
Then $M=\Gamma\backslash \tilde M$ is a compact locally symmetric
space. The geodesic flow and the frame flow descend to 
flows on the unit tangent bundle $T^1M$ of $M$ and 
on the bundle of $\mathbb{K}$-frames over $T^1M$ which 
are again denoted by $\Phi^t,\Psi^t$, respectively. 

The Riemannian metric on $M$ naturally defines Riemannian
metrics on $T^1M$ and on the $\mathbb{K}$-frame bundle ${\cal F}$.
These metrics determine a Borel probability measure
$\lambda$ on ${\cal F}$ in the Lebesgue measure class 
which is invariant under the flow $\Psi^t$. The
following is a classical important result in representation
theory.

\begin{theorem}\label{mixing}
The frame flow $\Psi^t$ is exponentially mixing for the measure
$\lambda$.
\end{theorem}
\begin{proof} The statement of the theorem is equivalent to 
exponential decay of matrix coefficients for the lattice $\Gamma$
\cite{CHH88}. 
We refer to the book \cite{BM00} for an overview of this and related
results. 
\end{proof}

For completeness of exposition we explain the meaning of this
statement. Namely, for some $k\geq 2$ 
equip the space of smooth functions on 
${\cal F}$ with the $C^k$-norm $\Vert \,\Vert$ with respect to the
Riemannian metric. Then there is a number $\kappa>0$ so that for any
two functions $f,g$ of class $C^k$ we have
\[\vert \int (f\circ \Psi^t) gd\lambda-\int fd\lambda\int gd\lambda\vert
\leq e^{-\kappa t}\Vert f\Vert \,\Vert g\Vert/\kappa.\]

\subsection{A criterion for incompressibilty}

In this subsection we consider 
a closed Riemannian  manifold $M$ of nonpositive
sectional curvature 
with universal covering
$\tilde M$. 

Let $S$ be a closed oriented surface of genus
$g\geq 2$ equipped with a locally ${\rm CAT}(-1)$ geodesic metric.
We say that a continuous map 
$f:S\to M$ is \emph{incompressible} if 
$f_*:\pi_1(S)\to \pi_1(M)$ is injective, and we use this terminology also 
in the case that $S$ is a compact surface with boundary.

The fundamental group $\pi_1(S)$ of $S$ acts
on the universal covering $\tilde S$ of $S$ as a group of 
isometries. 
Since $S,M$ are $K(\pi,1)$ spaces, 
there is an $f_*$-equivariant map 
$F:\tilde S\to \tilde M$ which projects to $f$.  
Here $f_*:\pi_1(S)\to \pi_1(M)$ is the homomorphism induced by $f$.
We call such a map $F$ a 
\emph{canonical lift} of $f$.





\begin{proposition}\label{cannonmap}
Let $S$ be a closed oriented 
surface with a locally ${\rm CAT}(-1)$
geodesic metric, 
and let $f:S\to M$ be a Lipschitz map
with canonical lift $F:\tilde S\to \tilde M$.
The following are equivalent.
\begin{enumerate}
\item $f$ is incompressible.
\item $F$ is proper.
\end{enumerate}
\end{proposition}
\begin{proof} 
%
%
%

To show $(1)\Rightarrow (2)$ 
assume that $f_*$ is injective. 
Let $K\subset \tilde S$ be a compact fundamental domain for the 
action of $\pi_1(S)$.  Then $F(K)\subset \tilde M$ is compact. 
If $B\subset \tilde M$ is any compact set, then the set 
\[A=\{\psi\in \pi_1(M)\mid \psi(F(K))\cap B\not=\emptyset\}\] 
is finite, and by equivariance, 
\[F^{-1}(B)\subset \cup \{\gamma(K)\mid f_*\gamma\in A\}.\]
Since $f_*$ is injective and $A$ is finite, this means that
the closed set
$F^{-1}(B)\subset \tilde S$ 
is contained in finitely many translates of $K$ and hence
$F^{-1}(B)$ is compact. Therefore 
$F$ is proper as claimed. 

To show the 
implication $(2)\Rightarrow (1)$ assume that
$F$ is proper. Assume to the contrary that there is
an element $0\not=\alpha\in \pi_1(S)$ with $f_*\alpha=0$.
Represent $\alpha$ by a closed geodesic in $S$, again
denoted by $\alpha$. The loop $f(\alpha)\subset M$ is 
contractible. Then $f(\alpha)$ lifts to a compact loop 
$\beta$ in $\tilde M$. The loop $\beta$ is the
image under $F$ of a lift $\tilde \alpha$ of $\alpha$ to $S$.
As $\tilde \alpha$ is unbounded, $F^{-1}(\beta)\supset 
\tilde \alpha$ is not compact. The proposition follows.
\end{proof}



\section{Spaced laminations}\label{spaced}

The goal of this section is to establish some distance estimates
on closed 
surfaces of genus $g\geq 2$ equipped with a 
locally ${\rm CAT}(-1)$ geodesic metric
with some
specific properties. 

The immersed surfaces 
in closed rank one locally symmetric spaces 
we are going to construct in Section \ref{theglueing}
will have all these properties.
They are glued 
from smooth pieces with geodesic boundary, where the smooth
pieces are equipped with 
a  smooth metric of Gauss curvature at most $-1/2$.
We begin with showing that the metric on such a surface
is locally ${\rm CAT}(-1/2)$. Lemma \ref{localcat} below 
is certainly known to the experts. As we were not
able to find a reference in the literature, 
we give a proof. For ease of exposition, we rescale and
work with locally ${\rm Cat}(-1)$-metrics.

\begin{lemma}\label{localcat}
Let $S$ be a surface equipped  
with a length metric $d$, with (perhaps empty) geodesic
boundary.
Assume that $S$ contains a compact embedded
geodesic  
graph $Q$ such that
the restriction of the metric $d$ to each component of $S-Q$ 
is a smooth Riemannian metric of curvature at most $-1$.
Assume moreover that at each vertex of $Q$  
the cone angle is not smaller than $2\pi$.
Then $d$ is locally ${\rm CAT}(-1)$.
\end{lemma} 
\begin{proof} It suffices to show that every point $x\in S$ has a convex
neighborhood $U(x)$ so that the triangle comparison property
holds for triangles with vertices in $U(x)$ (see Proposition II.1.7 of
 \cite{BH99}).

This follows from standard comparison 
if $x$ is an interior point of a component of $S-Q$ 
where the metric is smooth. Let $x$ be an interior 
point of an edge $\zeta$ of $Q$.
Assume that $\zeta$ separates an open
contractible neighborhood $U$ of $x$
into halfplanes $W_1,W_2$ with smooth metric and geodesic boundary.

Let $y_i\in W_i$ and consider a geodesic triangle $T$ with vertices
$x,y_1,y_2$ and edges of minimal length. 
Since the edge $\zeta$ of the graph $Q$ is a geodesic 
for the metric $d$ and the metric
in $S-Q$ is of curvature at most $-1$, 
the side of $T$ connecting $y_1$ to $y_2$ intersects $\zeta$
in a single point $y$ provided that the distance between $y_1,y_2$ is 
sufficiently small. The triangle $T_i$ with vertices
$x,y,y_i$ $(i=1,2)$ 
satisfies the angle comparison property. In particular, the Aleksandrov
angle of $T_i$ at $y$ is not bigger than the 
comparison angle in the hyperbolic plane ${\bf H}^2$.

For $i=1,2$ let $\overline{T_i}$ be a comparison 
triangle in ${\bf H}^2$ whose side lengths coincide with the
side lengths of $T_i$. Assume that 
$\overline{T_1}\cap \overline{T_2}$ is a common side of $\overline{T_1},
\overline{T_2}$ of length $d(x,y)$. By 
the discussion in the previous paragraph,
the angle sum of the geodesic quadrangle 
$\overline{T_1}\cup\overline{T_2}$  
at the point $\hat y\in \overline{T_1}\cap \overline{T_2}$ corresponding to 
the common vertex $y$ of $T_1$ and $T_2$ is not smaller than $\pi$. 
Hyperbolic trigonometry now shows that in a comparison triangle 
$\overline{T}\subset {\bf H}^2$ 
for $T$, the distance between the points $\bar x,\bar y\in \overline{T}$ 
corresponding to the points $x,y\in T$ 
is not smaller than the distance between $x$ and $y$. 

Using comparison
for the Riemannian triangles $T_1,T_2$ with triangles in ${\bf H}^2$, 
we conclude that the distance
between the vertex $\bar x$ of $\overline{T}$ corresponding to $x$ and 
a point on the 
opposite opposite side of $\bar T$ is not smaller than the distance
between $x$ and the corresponding point on the side of $T$ opposite to $x$.
With the same argument, this estimate also holds true for distances
between the other vertices and points on the opposite sides.

Proposition II.1.7 of \cite{BH99} now shows that the 
path metric $d$ on $S$ is locally ${\rm CAT}(-1)$  
on the complement of the vertex set of $Q$.

Using the condition on cone angles,  
the same argument also holds true near a vertex of the graph.
This implies the lemma.
\end{proof}

As a consequence, a closed curve on a surface
with the properties in Lemma \ref{localcat} 
has a unique geodesic representative in its free homotopy
class. Moreover, local geodesics lift to curves in the
universal covering which realize the distance between their
endpoints.

The surfaces we construct
will be glued from pairs of pants with specific properties.
In the remainder of this section we discuss 
those properties that are used to establish incompressibility.

Let $P_0$ be such a pair of pants with geodesic
boundary. Any two boundary geodesics are connected
by a unique shortest geodesic arc. Such a geodesic arc
is called a \emph{seam} of $P_0$.
These seams decompose $P_0$ into 
hexagons with geodesic boundary.
The angles in the sense of Aleksandrov of these
hexagons are at least $\pi/2$. 
The endpoints of the
seams define two distinguished points on each boundary
geodesic of $P_0$. We call these points the 
\emph{feet} of the pair of pants. 

In the
cases we are interested in, these hexagons are all
right angled since the restriction of the metric
to a pair of pants is smooth in a neighborhood of the
seams. In the remainder of this section 
we will use this assumption to faciliate the
notations, although it is nowhere used in the arguments. 

Each component $\alpha$ of the pants
decomposition is contained in the boundary of two
(not necessarily distinct) pairs of pants $P_1,P_2$. 
These pairs of pants
are glued with an orientation reversing
isometry along $\alpha$. 
The feet of $P_1$ on $\alpha$ need not coincide with the
feet of $P_2$ on $\alpha$. We define the 
\emph{shear} of the component $\alpha$ of ${\cal P}$ 
to be the pair of distances on $\alpha$ between
the feet of $P_1,P_2$ which are determined by the
orientations of $\alpha$ as boundary components of $P_1,P_2$.

More precisely, choose an endpoint $x\in \alpha$ of a seam of $P_1$ 
on $\alpha$. The orientation of $P_1$ defines an orientation
of $\alpha$. The oriented distance between $x$ and 
a seam of $P_2$ is the distance along $\alpha$ 
between $x$ and the first point $y$ on $\alpha$ which is
a seam of $P_2$. This distance equals the oriented distance
between $y$ and a seam of $P_1$ provided that
the second seam of $P_1$ does not lie between $x$ and $y$
along the oriented subarc of $\alpha$ connecting
$x$ to $y$.
Since the latter property holds true in the situations we are interested in 
we will always assume in the sequel that this is the case.

If the metric on $S$ is smooth and of constant 
curvature then 
the seams decompose each boundary geodesic into two arcs
of equal length and 
the two shear
parameters coincide. However, this need not be the case if the 
curvature is non-constant.
We say that the shear of the pants curve $\alpha$ 
is \emph{contained in an interval $[a,b]$} only if
both shear parameters in the pair are contained 
in this interval. 
In this vain, the following definition 
(which is motivated by the work of 
Kahn and Markovic \cite{KM12})
is natural for 
hyperbolic surfaces but harder to accomplish for surfaces
with arbitrary locally ${\rm CAT}(-1)$ geodesic metrics.

\begin{definition}\label{tightpants}
For some $\delta\in (0,1/4)$, $R>10$, 
a pants decomposition ${\cal P}$ of $S$ is 
\emph{$(R,\delta)$-tight} if 
the following holds true,
\begin{enumerate}
\item[$\bullet$]
The lengths of the pants curves
are contained in the interval $[R-\delta,R+\delta]$.
\item[$\bullet$]
The seams of a pair of pants decompose the boundary
geodesics into two subarcs whose lengths are contained 
in the interval $[R/2-\delta,R/2+\delta]$.
\item[$\bullet$]
The shear of each component of ${\cal P}$ is contained in the 
interval
$[1-\delta,1+\delta]$.
\end{enumerate}
\end{definition}

An $(R,\delta)$-tight pants decomposition of $S$ 
lifts to a discrete $\pi_1(S)$-invariant geodesic
lamination $\mu$ on the universal covering
$\tilde S$ of $S$. 
Figure A shows a geometric model 
for the geodesic lamination defined by an $(R,\delta)$-tight
pants decomposition.

\begin{figure}[ht]
\begin{center}
\includegraphics[width=0.7\linewidth]{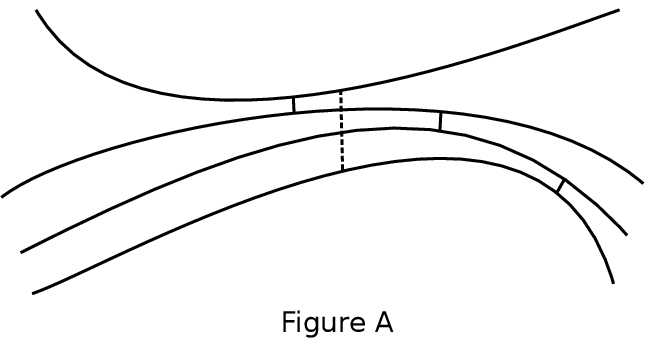}
\end{center}
\end{figure}

We will need some additional geometric information 
on the pants decompositions we are going to use. Namely,
for a number $C>0$  
call an $(R,\delta)$-tight  pants decomposition ${\cal P}$ of 
$S$ \emph{centrally $C$-thick} if 
the following holds true. 
Let $\beta:[0,s]\to S$ be a geodesic arc with endpoints on ${\cal P}$.
Assume that $\beta$ intersects
some pants curve $\alpha$ of ${\cal P}$ at a point whose
distance along $\alpha$ is at least
$R/4-2$ from the endpoint of a seam on $\alpha$.
Then the length of $\beta$ is at least $C$.

\begin{figure}[ht]
\begin{center}
\includegraphics[width=0.7\linewidth]{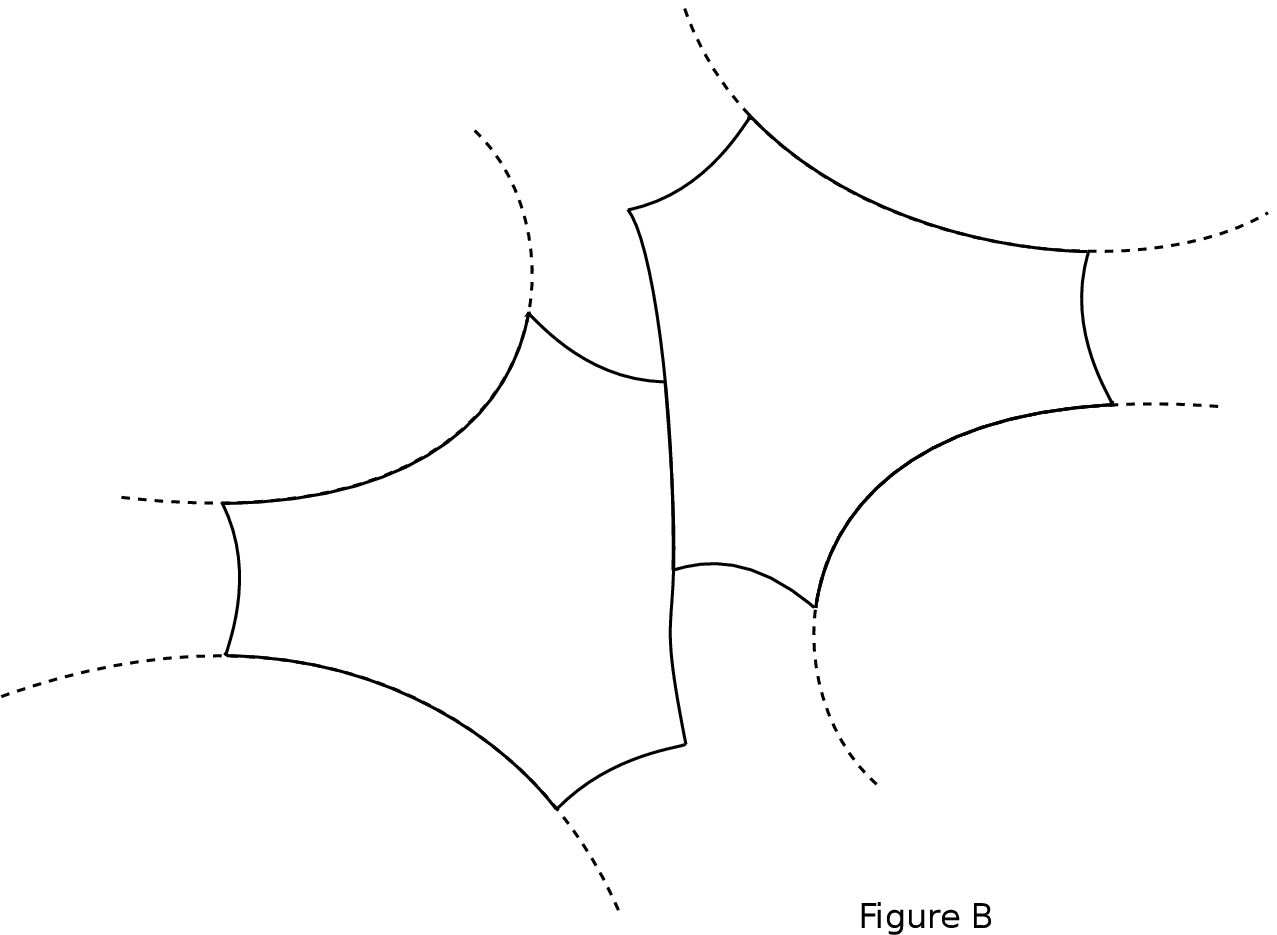}
\end{center}
\end{figure}

The following technical lemma uses the property described in 
this definition in an essential way. 
For its formulation, 
let ${\cal P}$ be an $(R,\delta)$-tight pants decomposition
of a surface $S$ equipped with a 
locally ${\rm CAT}(-1)$-metric.
For each pants curve $\alpha$ and every $x\in \alpha$ define
$\tau(x)$ to be the maximum of one and the 
distance of $x$ along $\alpha$ 
to the endpoint of a seam of ${\cal P}$ on $\alpha$. 
We then can view $\tau$ as a function on 
${\cal P}$ with values in the interval $[1,R/2+\delta]$. 
Let $b>1$.
Let $\zeta:[0,T]\to S$ be any geodesic segment
which is transverse to ${\cal P}$ 
and intersects ${\cal P}$ in the points
$\zeta(t_i)$ $(1\leq i\leq m)$.
Define
\[f(\zeta)=\sum_{\zeta(t_i)\in {\cal P}}\frac{1}{\tau(\zeta(t_i))^b}.\]

The estimate in the following lemma 
(which is a variant of a construction in \cite{KM12}) is
used in the proof of Lemma \ref{anglecontrol} which
is the main technical tool of this work.

\begin{lemma}\label{longarcs}
For every $C>0$ there is a number $\chi>0$ 
with the following property.
Let $R>10$, $\delta<1/10$ and let  
$S$ be a closed oriented surface of genus $g\geq 2$ equipped with
a locally ${\rm CAT}(-1)$-metric and an
$(R,\delta)$-tight centrally $C$-thick pants decomposition.
Then $f(\zeta)<\chi$ for every geodesic arc 
$\zeta$ on $S$ of length smaller than
$\min\{1/4,C\} $ which is transverse to $P$.
\end{lemma}
\begin{proof} Let $\tilde S$ be the universal covering of $S$.
The pants decomposition ${\cal P}$ lifts to a geodesic lamination
$\mu$ on $\tilde S$. The lifts of the seams of ${\cal P}$ decompose
the complementary components of $\mu$ into 
right angled hexagons. These hexagons define a tesselation of 
$\tilde S$ which is invariant under the action of $\pi_1(S)$.
Call a lift to $\tilde S$ of a seam of ${\cal P}$ a \emph{seam} of $\mu$. 
Let $\tau:\mu\to [0,\infty)$ be the function which 
associates to a point $x$ on a leaf $\alpha$ 
of $\mu$ the maximum of one and the 
minimal distance between
$x$ and an endpoint on $\alpha$ of some seam.

Let $\zeta:[0,a]\to \tilde S$ 
be a geodesic arc parametrized by arc length 
which is transverse to $\mu$ and such that
$\zeta(0)\in \mu, \zeta(a)\in \mu$. 
Assume that the length $a$ of $\zeta$ does not exceed
${\rm min}\{1/4,C\}$.

Let $0=t_0<t_1<\dots <t_m=a$ be the consecutive
intersection points of $\zeta$ with $\mu$. 
Assume for the moment that 
$\zeta[t_0,t_1]$ does not intersect a seam.
This is equivalent to stating that $\zeta(t_0,t_1)$ is 
contained in the interior of a hexagon $H_0$ of the 
$\pi_1(S)$-invariant tesselation and cuts $H_0$ into 
a quadrangle $Q_0$ and a hexagon $H_0-Q_0$.
The side $\xi$ of $Q_0$ opposite to $\zeta[t_0,t_1]$ 
is a seam. The quadrangle $Q_0$ has two right angles
at the endpoints of $\xi$. The sides of $Q_0$ adjacent
to $\zeta[t_0,t_1]$ 
are subarcs of leaves of $\mu$. 

Consider first the case that 
with respect to the orientation of $\tilde S$, the 
orientation of $\zeta$ defines the boundary orientation of $Q_0$.
Then the side $\xi$ of $Q_0$ is to the left of $\zeta$.
Since the length of $\zeta$ is smaller than $C$, 
by the definition of an $(R,\delta)$-tight 
centrally $C$-thick pants decomposition, we have
$\tau(\zeta(t_0))<R/4-2$. 

Now there are two possibilities.
In the first case, 
$\tau(\zeta(t_0))>1$. Since $\delta <1/10$ 
and since the length of 
$\zeta$ does not exceed $1/4$, 
the value
$\tau(\zeta(t_1))$ is the distance between 
$\zeta(t_1)$ and the side $\xi$ of the quadrangle $Q_0$.
Moreover, we have 
$\tau(\zeta(t_1))\geq \tau(\zeta(t_0))+1/2$.
In the second case, we have 
$\tau(\zeta(t_0))=1$. This means that the
distance along the leaves of $\mu$ between $\zeta(t_0)$ and 
a seam on the leaf of $\mu$ containing
$\zeta(t_0)$ is at most one.
The argument from the first case now shows that 
$\tau(\zeta(t_3))\geq \tau(\zeta(t_2))+1/2$.

Proceeding inductively and using the definition of a
centrally thick pants decomposition and the assumption 
on the length of $\zeta$, 
we conclude that 
\[f(\zeta)\leq 
\sum_{i=0}^{m}\frac{1}{(\tau(\zeta(t_0))+i/2)^b}+2\]
where $m$ is the smallest integer larger than $R/2$.  
This shows the lemma for geodesic arcs $\zeta$ 
with the following property. The arc $\zeta$ 
does not intersect a seam and
moreover,
with respect to a fixed orientation of $\zeta$,
all seams which are closest to $\zeta$ along the leaves of 
$\mu$ crossed through by $\zeta$ are
to the left of $\zeta$ as described above, i.e. the 
quadrangles  defined by the intersection of $\zeta$ with
the interiors of the hexagons from the tesselation lie to the 
left of $\zeta$. 

If with respect to a fixed orientation of $\zeta$,
some of the quadrangles defined by the intersection
of $\zeta$ with the interiors of the hexagons
from the tesselation lie to the right of $\zeta$ and 
some others lie to the left, then
we can decompose $\zeta$ into two disjoint
subarcs to which the above discussion can be applied.
The same holds true if $\zeta$ intersects a seam.
The lemma follows.
\end{proof}

\begin{remark}\label{onlylam}
Lemma \ref{longarcs} is also valid for 
an arbitrary discrete geodesic lamination on a simply connected
${\rm CAT}(-1)$-surface so that there is a system of 
shortest distance arcs between neighboring leaves of the
lamination with the properties described in the definition
of a tight pants decomposition. 
\end{remark}

Hyperbolic trigonometry implies that 
$(R,\delta)$-tight pants decompositions of hyperbolic
surfaces are centrally $C_0$-thick for a universal number
$C_0>0$. We formulate this as a lemma.

\begin{lemma}\label{realhyp}
There are numbers $C_0>0,R_0>0,\delta_0>0$ such that
for all $R>R_0,\delta<\delta_0$, every
$(R,\delta)$-tight pants decomposition of a 
surface of constant curvature $-1$ is centrally $C_0$-thick.
\end{lemma}
\begin{proof} 
Let $S$ be a hyperbolic surface equipped with 
an $(R,\delta)$-tight pants decomposition ${\cal P}$.
The universal covering of $S$ is the hyperbolic plane
${\bf H}^2$.

Each component $X_0$ of $S-{\cal P}$ 
is a union of two isometric right angled hyperbolic 
hexagons which are obtained by cutting $X_0$ open along the 
seams. The length of a long side 
of such a hexagon equals half the length of the component 
of ${\cal P}$
containing it, i.e. it equals 
$R/2$ up to an additive error of at most $\delta/2$.
Hyperbolic trigonometry (Theorem 2.4.1 of \cite{B92}) 
shows that  
the length of a side which corresponds to a seam is comparable to 
$e^{-R/4}$.


Let $T\subset {\bf H}^2$ be an ideal hyperbolic triangle,
i.e. a geodesic triangle with vertices on the ideal boundary
$\partial {\bf H}^2$ of ${\bf H}^2$,
and let $\gamma$ be one of the sides of $T$.  
The shortest distance projection into $\gamma$ 
of the ideal vertex of $T$ opposite to $\gamma$ 
is a distinguished point $x$ on $\gamma$.
There is a number $c>0$ such that the distance
between any point on $\gamma$ whose distance to $x$ is at most
three and a side of $T$ distinct from $\gamma$ is 
at least $c$.

The group $PSL(2,\mathbb{R})$ 
of orientation preserving isometries of ${\bf H}^2$
acts transitively on oriented ideal triangles. This implies that 
as $R\to \infty$ and $\delta\to 0$, a 
right angled hyperbolic hexagon $H$ with 
three pairwise non-adjacent sides of length
within $[(R-\delta)/2,(R+\delta)/2]$ 
converges up to the action of $PSL(2,\mathbb{R})$ in the
Hausdorff topology for closed subsets of the 
closed unit disc ${\bf H}^2\cup \partial{\bf H}^2$
to an ideal triangle. 
As a consequence, there are numbers 
$R_0>0,\delta_0>0$ such
that for $R\geq R_0$ and 
$\delta<\delta_0$ an $(R,\delta)$-tight 
pants decomposition of a hyperbolic surface is 
centrally $c/2$-thick.
\end{proof}

\section{Constructing geometrically controlled pants}
\label{constructing}

In this section we construct topological versions
of the pairs of pants which form the 
basic building blocks for our surfaces.
More precisely, for a given closed rank one locally 
symmetric manifold $M$,
we construct maps from a fixed 
pair of pants into $M$ which map 
the boundary circles of the pair of pants to closed
geodesics in $M$. 

The underlying principle for this construction is very
general and applies to any closed negatively curved
manifold with a generic metric (so that the frame flow is
topologically mixing). As we will need to obtain a good
geometric control for these pairs of pants we will however
only work in rank one locally symmetric manifolds.
Control of these geometric invariants in the construction 
is the only part of the argument in this section which
is not taken from \cite{KM12}.

We begin with a geometric construction in the 
hyperbolic plane ${\bf H}^2$. Define a \emph{tripod}
in ${\bf H}^2$ to be an ordered  triple 
$(v_1,v_2,v_3)$ of unit tangent vectors over a 
fixed point $x$ which mutually enclose an 
angle of $2\pi/3$. The tripod defines
an oriented ideal hyperbolic triangle $T$ whose endpoints 
in the ideal boundary $\partial {\bf H}^2$ of ${\bf H}^2$ 
are the endpoints of the geodesic rays $\gamma_{v_i}$ with 
initial velocity $v_i=\gamma_{v_i}^\prime(0)$. 
The orientation of $T$ is defined
by the cyclic order of the vertices $v_i$.

We call the basepoint $x$ of the tripod 
the \emph{center} of the triangle $T$.
The oriented boundary of $T$ is denoted by $\partial T$.
Note that $T$ is preserved by
the cyclic subgroup $\Lambda$ of $PSL(2,\mathbb{R})$ 
of order $3$ which fixes the point $x$ and which acts by
rotation with angle $2\pi/3$ in the tangent plane 
of ${\bf H}^2$ at $x$.

For $R>1$ let $H_{R}\subset T$ be 
the intersection of $T$ with
the half-planes containing $x$ whose boundaries are the 
geodesics through $\gamma_{v_i}(R)$ which are 
perpendicular to $\gamma_{v_i}$. 
Then $H_R$ is a $\Lambda$-invariant oriented
hyperbolic hexagon which is 
not right-angled.
Three sides of $H_{R}$ are contained in 
the sides of the ideal triangle 
$T$, and these sides are called the
\emph{long sides}. The length of each long side equals 
\[L(R)=2R+t(R)\] where
$t(R)\in (-\infty,0)$ 
is uniformly bounded in norm (recall that we require $R>1$.
The number $t(R)$ can explicitly be computed using 
the formulas of hyperbolic trigonometry, see \cite{B92}).
The length of the \emph{short sides} of $H_R$  
(i.e. the three sides which 
intersect the geodesics $\gamma_{v_i}$)
does not exceed $\kappa_0e^{-R}$ where
$\kappa_0>0$ is a universal constant. 
We call the footpoint $x$ of the tripod the \emph{center}
of the hexagon $H_R$ (see Figure C).

Resuming the notations from Section \ref{cannon},
let $G$   
be a simple rank one Lie group of non-compact type and let
$\Gamma<G$ be a torsion free cocompact lattice. 
Let $K<G$ be a maximal compact subgroup, let $\tilde M=G/K$ 
be the corresponding symmetric space and let
$M=\Gamma\backslash \tilde M$ be the locally symmetric space
defined by $\Gamma$.
We always assume that $M$ is equipped with the locally symmetric
metric whose upper curvature bound is $-1$. By 
perhaps passing to a subgroup of 
$\Gamma$ of index 2 we may assume that $M$ is oriented.

Define a \emph{real tripod} in 
$T\tilde M$ (or $TM$)
to be an ordered  triple $(v_1,v_2,v_3)$ of three unit tangent vectors
contained in the same real plane $V\subset T\tilde M$ 
(or $V\subset TM$)
which mutually enclose an angle of $2\pi/3$.
The cyclic order of the tripod defines an orientation of 
$V$. 


A real plane $V\subset T\tilde M$ is tangent to a unique oriented
totally geodesic embedded hyperbolic 
plane $H\subset \tilde M$. Thus a real 
tripod $(\tilde v_1,\tilde v_2,\tilde v_3)$ in 
$T\tilde M$ defines
an oriented real ideal triangle $T$ in the
hyperbolic plane $H\subset \tilde M$ containing the 
tripod in its tangent plane. 
The group $G$ acts
transitively on these oriented real ideal triangles.
 
For each $R>1$, a real tripod $(\tilde v_1,\tilde v_2,\tilde v_3)$ 
in $T\tilde M$ 
determines an oriented hexagon 
\[H_R(\tilde v_1,\tilde v_2,\tilde v_3)\] in 
the totally geodesic hyperbolic plane
$H\subset \tilde M$ tangent to the tripod. The hexagon
$H_R(\tilde v_1,\tilde v_2,\tilde v_3)$ 
is isometric to the hexagon $H_R
\subset {\bf H}^2$ described above. 
We use the terminology which was introduced
above for hexagons in ${\bf H}^2$ also 
for hexagons in $\tilde M$ defined by real tripods.

Define a \emph{framed real tripod} 
in $T\tilde M$ to be a pair of the form
$((\tilde v_1,\tilde v_2,\tilde v_3),F)$ 
where $(\tilde v_1,\tilde v_2,\tilde v_3)$ is a real tripod
contained in a real plane $V\subset T\tilde M$ 
and where $F$ is a positive $\mathbb{K}$-orthonormal frame 
in the orthogonal complement of the $\mathbb{K}$-span 
$V\otimes \mathbb{K}$ of 
$V$. Here the orientation of $V$
is determined by the tripod.  
The group $G$ acts on framed real tripods and therefore 
we can also consider framed real tripods in $TM$.
In fact, the action of $G$ on such framed real tripods is
simply transitive, so framed real tripods in $T\tilde M$  
can be viewed as points in $G$. However, we will use the specific
geometric meaning of framed real tripods, moreover the geometric
discussion is valid in any closed oriented negatively curved manifold.

A framed 
real tripod $((\tilde v_1,\tilde v_2,\tilde v_3),F)$ 
in $T\tilde M$ determines for each $R>1$ a \emph{framed hexagon} 
$(H_R(\tilde v_1,\tilde v_2,\tilde v_3),F)$. Namely, 
for $i=1,2,3$, the frame $F$ determines 
frames 
 \[F_i\to \gamma_{\tilde v_i}^\prime(R)\to \gamma_{\tilde v_i}(R)\]
in the fibre of the bundle 
${\cal F}$ at the point $\gamma_{\tilde v_i}^\prime(R)$
as follows. 
Choose the first vector of the frame $F_i$  
to be the oriented normal of $\gamma_{\tilde v_i}^\prime(R)$ 
in the oriented hyperbolic 
plane $H\subset \tilde M$ defined by the tripod. 
The remaining ordered vectors of the frame are obtained by parallel
transport of the frame $F$ along $\gamma_{\tilde v_i}$.

By Lemma \ref{paralleltransport},
for all $i,j$ the complement of the first vector
of the frame $F_j$ can also be obtained from 
the complement of the first vector 
of the frame $F_i$ by parallel transport along the boundary of 
$H_{R_i}(\tilde v_1,\tilde v_2,\tilde v_3)$. Moreover, 
the first vector of $F_i$ is 
uniquely determined by the tripod $(\tilde v_1,\tilde v_2,
\tilde v_3)$ and the size parameter $R$.


Each real tripod
$(\tilde v_1,\tilde v_2,\tilde v_3)$ in $T\tilde M$ 
projects to a real tripod
$(v_1,v_2,v_3)$ in $TM$, and the 
hexagon $H_R(\tilde v_1,\tilde v_2,\tilde v_3)\subset \tilde M$ projects 
to a totally geodesic immersed 
hexagon $H_R(v_1,v_2,v_3)$ 
in $M$ which is uniquely determined by the tripod $(v_1,v_2,v_3)$ and 
the size parameter $R$.
Since the stabilizer in $\Gamma=\pi_1(M)$ of the hyperbolic plane
$H$ tangent to $(\tilde v_1,\tilde v_2,\tilde v_3)$ 
may be non-trivial, the hexagon 
$H_R(v_1,v_2,v_3)$ may have
non-transverse self-intersections. However, the projection
$H_R(\tilde v_1,\tilde v_2,\tilde v_3)\to 
H_R(v_1,v_2,v_3)$ is an isometric immersion.
The geodesics  
$\gamma_{v_i}$ with initial velocity $v_i$ 
will be called the \emph{center geodesics}
of the hexagon $H_R(v_1,v_2,v_3)$. Their initial segments
of length $R$ are contained in $H_R(v_1,v_2,v_3)$.

For a number $\epsilon >0$ define 
an element $A\in SO(m)$ ($m={\rm dim}(\tilde M)$)
to be \emph{$\epsilon$-close to the 
identity} if for every unit vector $v$ the angle between
$v$ and $Av$ is at most $\epsilon$.

The following definition is taken from 
the beginning of Section 4 of \cite{KM12}.
For its formulation, we say that the angle
between two planes $E_1,E_2\subset T_yM$ 
in the tangent space of $M$ at some
point $y$ is at most $\epsilon$ if for each 
vector $0\not=v$ in $E_i$ there is a vector 
$0\not=v^\prime\in E_{i+1}$ with $\angle(v,v^\prime)<\epsilon$
(indices are taken modulo two).

\begin{definition}\label{wellconnected}
For $R>0,\delta>0,\kappa>0$ call the framed 
tripods $((v_1,v_2,v_3),E)$ and
$((w_1,w_2,w_3),F)$ in $TM$ 
\emph{$(R,\delta,\kappa)$-well connected} if 
for each $i=1,2,3$ there is a geodesic segment
$\alpha_i$ connecting $\gamma_{v_i}(R)$ to 
$\gamma_{w_{-i}}(R)$ (indices are taken modulo three) so that 
the following holds true.
\begin{enumerate}
\item The length of $\alpha_i$ is contained in the interval
$[4R-\delta/4,4R+\delta/4]$.
\item The breaking angles of the concatenation
\[\gamma_{w_{-i}}^{-1}\circ\alpha_i\circ \gamma_{v_i}\] 
at the points
$\gamma_{v_i}(R),\gamma_{w_{-i}}(R)$ 
are not bigger than $e^{-\kappa R}/\delta$.
\item Let $E_i$ (or $F_i$) be the frame over 
$\gamma_{v_i}(R)$ (or over $\gamma_{w_{-i}}(R)$) 
defined by the framed tripod as above. 
Let $\hat E_i$ be the parallel transport of $E_i$ along $\alpha_i$.
Then
the element of $SO(m)$ $(m={\rm dim}(M))$ 
which transforms the frame $\hat E_i$ 
to the frame over $\gamma_{w_{-i}}(R)$ which is obtained from 
$F_i$ by replacing the first vector by its negative
is $\delta$-close to the identity.
\end{enumerate}
The geodesics $\alpha_i$ are called \emph{good connections}
for the tripods.
\end{definition}

By property (2), the angle between 
$-\gamma_{w_{-i}}^\prime(R)$ and the parallel transport of 
$\gamma_{v_i}^\prime(R)$ along $\alpha_i$ does not exceed
$2e^{-\kappa R}/\delta$.
The third requirement implies that 
the angle between the following two 
real planes in $T_{\gamma_{w_{-i}}(R)}M$ 
is at most $\delta$:
\begin{enumerate}
\item[$\bullet$] The tangent plane of the totally geodesic
hyperbolic plane containing the hexagon 
$H_R(w_1,w_2,w_3)$. 
\item[$\bullet$] The image under parallel transport
along $\alpha_i$ of the tangent plane 
of the hyperbolic plane containing the
hexagon
$H_R(v_1,v_2,v_3)$.
\end{enumerate}

\begin{remark}\label{extra}
For the purpose of this work,
the constant
$\kappa>0$ plays no role- it 
is geared at treating the case of 
locally symmetric spaces of higher rank.
We will only work with $(R,\delta,1)$-well connected
tripods which we call $(R,\delta)$-well connected
in the sequel.
\end{remark}

Recall that 
there is a natural Riemannian
metric on the bundle ${\cal F}$ characterized by
the property that the projection
${\cal F}\to T^1M$ is a Riemannian submersion with 
homogeneous fibre isometric to  
$SO(n-1)$ (or $SU(n-1)$ or $Sp(n-1)$ or
${\rm Spin}(7)$).
We next give a criterion for framed tripods to
be $(R,\delta)$-well connected.

Let $((v_1,v_2,v_3),E),((w_1,w_2,w_3),F)$ be two framed tripods in 
$TM$ and let $R>10,\delta \in (0,1/4)$. 
The framed tripods define frames $V_i\in {\cal F},
W_i\in {\cal F}$ over the tangent vectors
$\gamma_{v_i}^\prime(2R),\gamma_{w_i}^\prime(2R)$.

\begin{lemma}\label{framedconst}
Assume that for each $i$ there is a frame $V_i^\prime$
contained in the $\delta/2$-neighborhood of 
$V_i$ with the following property.
Let $PV_i^\prime\in T^1M$ be the 
base vector of the frame. Then 
the frame obtained from $\Psi^{2r}(V_i^\prime)$ by
replacing the base vector $\Phi^{2R}(PV_i^\prime)$
as well as the first vector of $\Psi^{2R}(V_i^\prime)$
by their negatives is contained in the
$\delta/2$-neighborhood of $W_{-i}$. Then the framed tripods
together with the frames $V_i^\prime$ determine an
$(R,\delta)$-well connected pair of framed tripods
provided $\delta>0$ is sufficiently small.
\end{lemma}
\begin{proof} Using the notation from the lemma,
construct a piecewise geodesic $\hat \alpha_i$ connecting
$\gamma_{v_i}(R)$ to $\gamma_{w_{-i}}(R)$ as a 
concatenation of the following
geodesic arcs.
\begin{enumerate}
\item[$\bullet$] 
$\gamma_{v_i}[R,2R]$,
\item[$\bullet$] an arc of length 
at most $\delta/2$ connecting 
$\gamma_{v_i}(2R)$ to $\beta_i(0)$,  
\item[$\bullet$] the geodesic $\beta_i$, 
\item[$\bullet$] an arc of length
at most $\delta/2$ connecting $\beta_i(2R)$ to 
$\gamma_{w_{-i}}(2R)$,
\item[$\bullet$]
the inverse of $\gamma_{w_{-i}}[R,2R]$.
\end{enumerate}
Let $\alpha_i$ be the geodesic in $M$ which is
homotopic to $\hat \alpha_i$ with fixed endpoints.
The length of $\alpha_i$ is contained in the interval
$[4R-\delta,4R+\delta]$.
We claim that
the angle between $\alpha_i$ and $\gamma_{v_i}$,
$\gamma_{w_{-i}}^{-1}$ at the endpoints 
$\gamma_{v_i}(R),\gamma_{w_{-i}}(R)$ is at most 
$\kappa e^{-R}$ where $\kappa >0$ does not depend on $R$.

To this end choose lifts of the geodesics $\gamma_{v_i},
\gamma_{w_{-i}}, \beta_i$ to $\tilde M$, say geodesics
\[\tilde \gamma_{v_i},\tilde \gamma_{w_{-i}},\tilde \beta_i,\]
so that the distance between the tangents 
$\tilde \gamma_{v_i}^\prime(2R),\tilde \beta_i^\prime(0)$
and between the tangents $\tilde \beta_i^\prime(2R)$, 
$-\tilde \gamma_{w_{-i}}^\prime(2R)$ 
is at most $\delta$. Let $\xi_i$ be the geodesic 
connecting $\tilde \gamma_{v_i}(2R)$ to 
$\tilde \beta_i(R)$.
Hyperbolic trigonometry \cite{B92} 
and comparison \cite{CE75} 
shows that the angle between $\tilde \gamma_{v_i}^\prime(2R)$ and
the tangent of  $\xi_i$  at $\tilde\gamma_{v_i}(2R)$ 
is at most $c_0\delta$ where
$c_0>0$ is a universal constant.  Moreover, the
angle between $\tilde \beta_i^\prime$ and 
$\xi_i^\prime$ at $\tilde \beta_i(R)$ is at most
$c_0e^{-R}$.

Let $\zeta_i$ be the geodesic connecting 
 $\tilde \gamma_{v_i}(R)$ to 
$\tilde \beta_i(R)$. 
Use comparison for the geodesic triangle with vertices
$\tilde \gamma_{v_i}(R),\tilde \gamma_{v_i}(2R),\tilde \beta_i(R)$ 
to conclude that the angle at $\tilde \gamma_{v_i}(R)$
between $\tilde \gamma_{v_i}^\prime$ and the tangent of 
$\zeta_i$ is at most $c_1e^{-R}$ where again,
$c_1\geq c_0$ is a universal constant. The
angle at $\tilde \beta_i(R)$ between $\tilde \beta_i^\prime$
and the tangent of $\zeta_i$ does not exceed 
$c_1e^{-R}$ as well. 

Apply this reasoning to the geodesics $\tilde \gamma_{w_{-i}}$ and 
the inverse of $\beta_i[R,2R]$ to control the tangents at the endpoints
of the geodesic $\eta_i$ connecting $\tilde \gamma_{w_{-i}}(R)$ to 
$\tilde \beta_i(R)$. We find that the angle at $\tilde \beta_i(R)$ 
between $\zeta_i$ and the inverse of $\eta_i$ does not exceed
$c_2e^{-R}$ for a universal constant $c_2>0$.
Thus by triangle comparison, the angle at $\tilde \gamma_{v_i}(R)$ between
$\tilde \gamma_{v_i}^\prime$ and the tangent of the geodesic 
connecting $\tilde\gamma_{v_i}(R)$ to $\tilde \gamma_{w_{-i}}(R)$ is 
a most $c_2e^{-R}$. The above 
claim now follows from this and symmetry.

The statement about the parallel transport is derived in the same
way.
\end{proof}

Lemma \ref{framedconst} is the method 
for the construction of the building blocks for our surfaces, namely
pairs of pants immersed in $M$. 
In the remainder of this section 
we explain why it gives rise to pairs of pants. 
We also collect some first properties of these pairs of pants which will 
be used to get some geometric control as $R,\delta$ vary. 
In Section \ref{twisted} and Section \ref{skewpants}
we will determine suitable sizes for $R,\delta$ using this
a-priori geometric control to establish a sufficient
condition for incompressibility of surfaces glued from pants.



Let $((v_1,v_2,v_3),E)$ and $((w_1,w_2,w_3),F)$ be 
$(R,\delta)$-well connected framed tripods in $TM$ with foot-points
$p,q\in M$. Let as before 
$H_R(v_1,v_2,v_3)$ and $H_R(w_1,w_2,w_3)$ be the 
totally geodesic immersed hyperbolic hexagons defined by
these tripods. We use the notations as in the definition of 
well connected tripods.

For each $i$ let $\beta_i$ be the 
geodesic arc in the hexagon
$H_R(v_1,v_2,v_3)$ which connects
the point $\gamma_{v_i}(R)$ to the point 
$\gamma_{v_{i+1}}(R)$ and define in the same way a geodesic
arc $\eta_i$ in $H_R(w_1,w_2,w_3)$ 
connecting $\gamma_{w_i}(R)$ to 
$\gamma_{w_{i+1}}(R)$ as shown in Figure C.
\begin{figure}[ht]
\begin{center}
\includegraphics[width=0.5\linewidth]{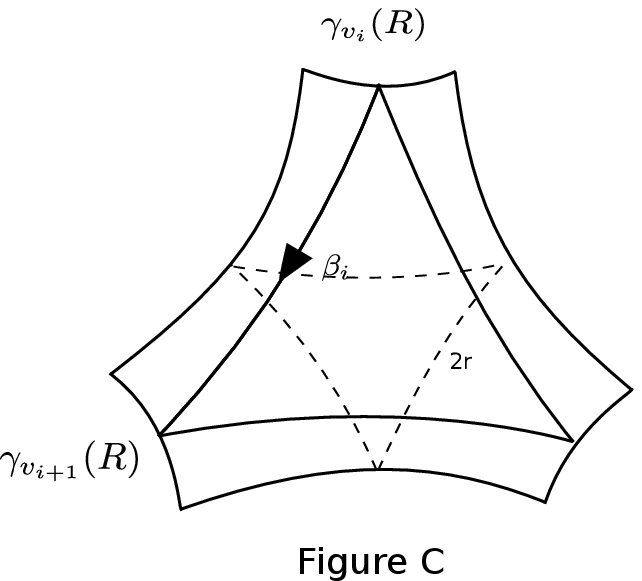}
\end{center}
\end{figure}

There is a number $q>\delta$ not depending on $R$ such that
the length $L(R)$ of these geodesics is contained in the
interval $[2R-q,2R]$.
Hyperbolic trigonometry shows that
the angle at $\gamma_{v_i}(R)$ of the triangle with 
vertices $p,\gamma_{v_i}(R),\gamma_{v_{i+1}}(R)$ is not bigger than
$\kappa_1e^{-R}$ 
where $\kappa_1>0$ is a universal constant.

Thus for each $i\in \{1,2,3\}$ 
the concatenation
\[\hat \gamma_i= \alpha_i^{-1}\circ \eta_i^{-1}\circ \alpha_{i+1}\circ
\beta_i\] 
(read from right to left and indices are taken modulo three) 
is a piecewise geodesic loop with $4$ breakpoints 
of breaking angle at most 
$\kappa_2e^{-R}$ where $\kappa_2>\kappa_1$ is a universal constant.
The lengths of the geodesic segments which form these
piecewise geodesic loops are at least $2R-\kappa_3$
where $\kappa_3>0$ is a universal constant.
The piecewise geodesic $\hat \gamma_i$ inherits from the boundary
orientation
of the oriented hexagons in the construction a natural orientation.

Standard comparison implies that
for sufficiently large $R$ the piecewise geodesic loop 
$\hat \gamma_i$ is freely homotopic
to a closed geodesic $\gamma_i$ in $M$.
The Hausdorff distance between the tangent line
of $\hat \gamma_i$ and the tangent line of $\gamma_i$ is 
at most $\kappa_4e^{-R}$ where once again, $\kappa_4>0$ does not depend on $R$.
By increasing $R$ we may assume that $\kappa_4e^{-R}<\delta/4$.
Then the lengths $\ell(\gamma_i)$ of the 
geodesics $\gamma_i$ satisfy
\[\ell(\gamma_i)\in 
[2L(R)+8R-\delta,2L(R)+8R+\delta].\] 
The geodesics $\gamma_i$ $(i=1,2,3)$ are pairwise
distinct, and there is an oriented pair of pants $P$ and an
incompressible map $f_P:P\to M$ which maps 
the boundary geodesics of $P$ 
onto the three geodesics $\gamma_1,\gamma_2,\gamma_3$.

The homotopy class of the map $f$ 
as well as the orientation of $P$ are
determined by the tripods $(v_1,v_2,v_3)$,
$(w_1,w_2,w_3)$ and the good connections $\alpha_i$. 
Note however that the tripods are not determined
by the homotopy class of $f$.
Following \cite{KM12} we call  $f(P)$ an 
\emph{$(R,\delta)$-skew pants}, and we identify two such
skew pants if they are defined by homotopic maps.

Let $\gamma$ be a boundary geodesic of a skew pants.
Then $\gamma$ is the quotient of a geodesic line $\tilde \gamma$
in $\tilde M$ under a hyperbolic isometry.
Such an isometry is determined by its translation length
(which is the length of $\gamma$) and a rotational 
part which is an element in $SO(n-1)$ for $G=SO(n,1)$,
an element in $SU(n-1)$ for $G=SU(n,1)$, an element in 
$Sp(n-1)$ for $G=Sp(n,1)$ and an element in ${\rm Spin}(7)$ for
$G=F_4^{-20}$. 
We call this rotational part 
the \emph{monodromy} of $\gamma$.
As before, there is a notion of being 
$\epsilon$-close to the identity for such a monodromy map.
The following proposition is immediate from the above construction
and from uniform continuity of parallel transport along 
piecewise geodesics.

\begin{proposition}\label{boundaryskew}
There is a number $\chi>0$ with the following property.
If $\gamma$ is a boundary curve of an $(R,\delta)$-skew-pants 
then
the monodromy of $\gamma$ is $\chi\delta$-close to the identity.
\end{proposition}

\section{Twisted bands}\label{twisted}

Consider again a rank one symmetric space
$\tilde M$ of curvature contained in the interval
$[-4,-1]$ and dimension at least three.
If the curvature of $\tilde M$ is constant then we require that
this constant equals $-1$.
The goal is to 
introduce a geometric model for the thin parts of the
pairs of pants constructed in Section \ref{constructing} 
in a compact quotient $M=\Gamma\backslash \tilde M$ of $\tilde M$.
Such a model is a 
twisted ruled band as defined below. The geometric realization
of an $(R,\delta)$-skew pants will consist of three ruled surfaces
which are exponentially close such twisted ruled bands.
These ruled surfaces are attached to two ruled geodesic triangles
which are exponentially close to the center triangle of an
ideal immersed hyperbolic triangle as defined in 
the next paragraph.

Let ${\bf H}^2$ be the hyperbolic plane and let 
$T\subset {\bf H}^2$ be an ideal hyperbolic triangle.
The projection of an ideal vertex of $T$ 
to the opposite side $\gamma$ 
is a special point on $\gamma$. The three special points
on the three sides of $T$ are the vertices of  
an equilateral hyperbolic triangle 
$T_0\subset T$ which we call the \emph{center triangle}.
Let $2r >0$ be the length of the sides of this triangle.
This length does not depend on $T$. The number $r$ will 
be used throughout the rest of this section.

Let $R\geq 10$ and let 
$\gamma:[-R,R]\to \tilde M$ be any geodesic arc of length $2R$.
Let 
\[V\to \gamma\] be the subbundle of the restriction of
$T\tilde M$ to $\gamma$ 
whose fibre at $\gamma(t)$ equals the $\mathbb{K}$-orthogonal
complement $(\gamma^\prime(t))_{\mathbb{K}}^\perp$ 
of $\gamma^\prime(t)$.
The bundle $V$ is 
invariant under parallel transport along $\gamma$.

Let $w_{-R}\in V_{\gamma(-R)},
w_R\in V_{\gamma(R)}$ be unit vectors. Let 
$\delta\in [0,\pi/4]$ and 
assume that the non-oriented 
angle between $w_R$ and the parallel transport
of $w_{-R}$ along $\gamma$ equals $\delta$. This
does not depend on the orientation of $\gamma$.


Let $\nu_{-R}$ and $\nu_R$ be the geodesic 
connecting $\exp(-rw_{-R})$ to $\exp(rw_{-R})$ and 
connecting $\exp(-rw_R)$ to $\exp(rw_R)$, respectively.
We assume that the geodesics $\nu_{-R},\nu_R$ are parametrized
by arc length on $[-r,r]$. Then 
\[\nu_{-R}(0)=\gamma(-R),\,\nu_R(0)=\gamma(R).\]
Moreover, $\nu_{-R},\nu_R$ meet $\gamma$
orthogonally at $\gamma(-R),\gamma(R)$.

Let $\ell>R$ be such that the distance
between $\nu_{-R}(r)$ and $\nu_R(r)$ equals $2\ell$.
For each $s\in [-r,r]$ connect $\nu_{-R}(s)$ to $\nu_R(s)$ by a geodesic
$\alpha_s$ parametrized proportional
to arc length on $[-\ell,\ell]$. Up to parametrization, 
we have $\alpha_{0}=\gamma$.
The map 
\[\alpha:[-r,r]\times [-\ell,\ell]\to \tilde M\] defined by 
$\alpha(s,t)=\alpha_s(t)$ is an embedding. Its image is a ruled
surface $E$ with induced orientation and oriented 
boundary $\alpha_{r}^{-1}\circ \nu_R\circ  \alpha_{-r}
\circ \nu_{-R}^{-1}$ (read from right to left).
We call this surface a \emph{$\delta$-twisted ruled band of size $2R$}
with central geodesic $\gamma$, and we call $\alpha_{-r},\alpha_r$ the
\emph{long sides} of the band. We also say that $E$ 
is a ruled band of twisting number $\delta$.
The map $\alpha$ is called the \emph{standard parametrization}
of the twisted ruled band $E$.

Let $x=\gamma(0)$ be the midpoint of $\gamma$ and let 
$\hat w_{-R},\hat w_R\in T_x\tilde M$ 
be the images of $w_{-R},w_R$ under 
parallel transport along $\gamma$.
The angle between $\hat w_{-R},\hat w_R$ 
equals $\delta$. Let $\hat w$ be the midpoint 
between $\hat w_{-R},\hat w_R$ in 
the fibre $V_x$ of $V$ over $x$, i.e. the
midpoint of the unique shortest arc in the unit sphere 
in $V_x$ connecting $\hat w_{-R},\hat w_R$. Let 
$\hat \nu$ be the geodesic in $\tilde M$ through $x$ which is 
tangent to $\hat w$. The geodesic $\hat \nu$ is
orthogonal to $\gamma$. There is a unique
totally geodesic hyperbolic plane 
$H\subset \tilde M$ of curvature $-1$
containing both $\gamma$ and $\hat \nu$.

\begin{lemma}\label{midpoint}
A subsegment of $\hat \nu$ is the 
unique geodesic arc in $\tilde M$ which is orthogonal
to both $\alpha_{-r},\alpha_r$, and it is contained in the ruled surface $E$.
\end{lemma}
\begin{proof} We begin with showing the lemma in the case 
$G=SO(n,1)$.
Then there is an isometric involution 
$\Psi$ of $\tilde M={\bf H}^n$ which 
fixes $\hat \nu$ pointwise
and whose differential 
acts as a reflection in the orthogonal complement of 
the tangent line of $\hat \nu$. 
The isometry $\Psi$ preserves the geodesic $\gamma$ and 
exchanges its endpoints. Moreover, we have
\[d\Psi(\hat w_{-R})=\hat w_R.\] 
Since isometries commute
with parallel transport, this implies that
$d\Psi(w_{-R})=w_R$. As a consequence,  
$\Psi$ preserves the ruled surface $E$ and 
acts as a reflection on the sides $\alpha_{-r},\alpha_r$. 
Since the fixed point set of $\Psi$ equals $\hat \nu$,  
there is a subarc $\nu$ of $\hat \nu$ which 
is contained in $E$. This subarc 
is the shortest geodesic  
between $\alpha_{-r}$ and $\alpha_r$, and it
meets the 
geodesics $\alpha_{-r},\alpha_r$ orthogonally at its endpoints.


Next consider the case that $\tilde M=\mathbb{C}{\bf H}^2$
equals the complex hyperbolic plane. With respect to 
a suitable choice of complex coordinates in the unit ball
in $\mathbb{C}^2$, complex conjugation is an 
anti-holomorphic isometry $\Theta$ 
of $\mathbb{C}{\bf H}^2$ which fixes the hyperbolic plane
$H$ containing $\gamma$ and $\hat \nu$ 
pointwise and acts as a reflection in the normal bundle of 
$H$. Thus we have
\[d\Theta(\hat w_{-R})=\hat w_R.\] 

Let $\sigma$ be the geodesic symmetry at $\gamma(0)$.
Then $\sigma\circ \Theta$ preserves both $\gamma$ and $\hat \nu$,
and it exchanges the endpoints of $\gamma$.
Moreover, we have  
$d(\sigma\circ \Theta)(\hat w_{-R})=-\hat w_R$.
Thus $\sigma\circ \Theta$ 
exchanges
the geodesics $\nu_{-R}$ and $\nu_R$. In particular, it
preserves the $\delta$-twisted ruled band $E$.
As before, this implies the statement of the lemma. 

If more generally $G=SU(n,1)$ for $n\geq 3$ then 
the real hyperbolic plane $H$ is contained in a unique
totally geodesic complex hyperbolic plane
$V=\mathbb{C}{\bf H}^2\subset \tilde M$. The anti-holomorphic
involution of $V$ 
which fixes $H$ pointwise and acts as a reflection in the normal bundle
of $H$ in $V$ can be extended to 
an anti-holomorphic isometry $\Theta$ of $\tilde M$
which fixes the point $\gamma(0)$ and 
maps the orthogonal projection of $\hat w_{-R}$ 
into $T_{\gamma(0)}V^\perp\subset T_{\gamma(0)}\tilde M$ 
to its negative, i.e. to the orthogonal projection of
$\hat w_R$.
The argument for the case $\tilde M=\mathbb{C}{\bf H}^2$ 
applies and yields the statement of the lemma in this case
as well.

The case $G=Sp(n,1), F_4^{-20}$ is completely analogous 
to the case $G=SU(n,1)$ and will be 
omitted.
\end{proof}

In the sequel we call the subsegment $\nu_0$ of the 
geodesic $\hat \nu$ as in Lemma \ref{midpoint} whose
endpoints are contained in the two long sides
of $E$ the \emph{seam} of $E$. This is consistent with 
the terminology used in Section 3.

\begin{figure}[ht]
\begin{center}
\psfrag{$\nu_{-R}$}{$\nu_{-R}$}
\psfrag{$nu_R$}{$\nu_{R}$}
\psfrag{$\nu_0$}{$\nu_0$}
\psfrag{$\gamma$}{$\gamma$}
\includegraphics[width=0.7\linewidth]{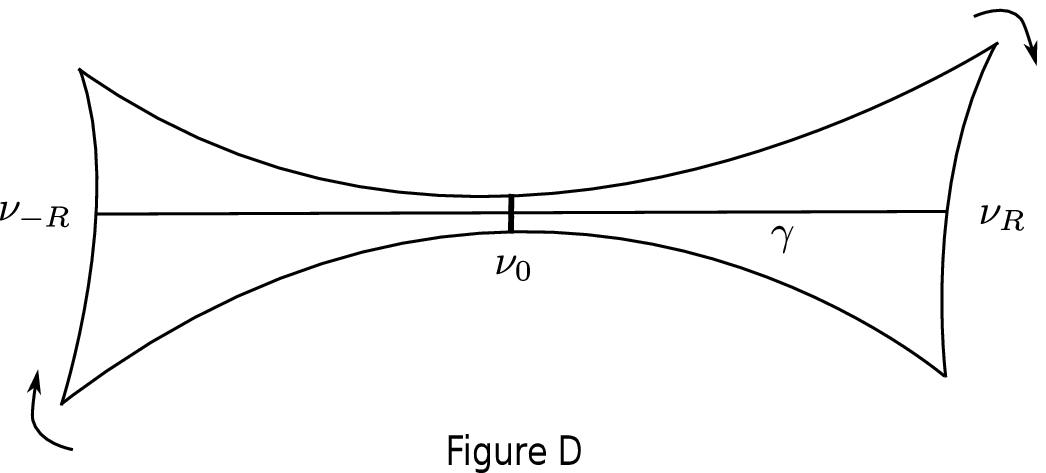}
\end{center}
\end{figure}

\begin{remark}\label{remarkone}
The proof of Lemma \ref{midpoint} also 
implies the following. 
\begin{enumerate}
\item  
For a $\delta$-twisted 
ruled band $E$, there is a unique totally geodesic
real hyperbolic plane $H(E)\subset \tilde M$ containing both
the central geodesic $\gamma$ and the seam $\nu_0$
of $E$. 
\item A $\delta$-twisted ruled band $E$ consists of two 
isometric copies of a ruled quadrangle $Q$. 
One side of $Q$ 
is the seam $\nu_0$ of $E$, with adjacent right angles.
The sides adjacent to $\nu_0$ 
have the same length, and the length of the opposite
side $\xi$ of $Q$ (which is a short side of $E$) is $2r$. 
The ruling consists of geodesic segments connecting the 
side $\nu_0$ to $\xi$.
The length of the side $\nu_0$ 
is contained in the interval 
$[c^{-1}e^{-R},ce^{-R}]$ for a universal constant 
$c>0$. Note that this estimate also holds true if the 
curvature of $\tilde M$ is not constant.
\end{enumerate} 
\end{remark}


Let 
\[{\cal V}\to \tilde M\]
be the bundle of oriented $2$-planes in 
$T\tilde M$. Its fibre 
over a point $x\in \tilde M$ is the Grassmannian of all
oriented two-dimensional linear subspaces  of $T_x\tilde M$.
The symmetric Riemannian 
metric of $\tilde M$ naturally induces
a Riemannian metric on ${\cal V}$ so that
${\cal V}\to \tilde M$ is a Riemannian submersion, with fibre
isometric to a compact symmetric space.
Denote by $d_{\cal V}$ the induced distance function on
${\cal V}$.

If $H\subset \tilde M$ is an oriented totally geodesic real
hyperbolic plane, then the oriented tangent bundle 
\[{\cal T}(H)\] of $H$ is naturally a totally geodesic 
submanifold of 
${\cal V}$. The projection ${\cal T}(H)\to H$ is an isometry.
Moreover, there 
is a unique shortest distance projection
\[\Pi_{H}:\tilde M\to H.\]

Let $E\subset \tilde M$ be a $\delta$-twisted ruled band of 
size $2R$, with central geodesic $\gamma$ and long sides 
$\alpha_{-r},\alpha_r$. 
By Remark \ref{remarkone},
there is a unique
real hyperbolic plane $H(E)\subset \tilde M$ which 
contains $\gamma$ and the 
seam $\nu_0$ of $E$.
Note that for $\delta=0$ the band $E$ is contained
in $H(E)$.

We next use the shortest distance
projection $\Pi_{H(E)}$ to investigate the 
geometry of a twisted ruled band $E$. 
To this end note that for every oriented long side 
$\beta$ of a $\delta$-twisted ruled band $E$ there is
a natural oriented plane field 
\[V(\beta,E)\to \beta\] 
whose fibre at a point
$\beta(t)$ is spanned by $\beta^\prime(t)$ and 
the parallel transport along $\beta$ of the tangent of the 
seam $\nu_0$ at the midpoint $\beta(0)$ of $\beta$, oriented
as the inner normal of the band.
Note that if $\tilde M$ is a real hyperbolic space
then this plane field is tangent to a totally geodesic hyperbolic plane
embedded in $\tilde M$,
but this need not be the case in general.
We have

\begin{lemma}\label{planedist}
There is a number $C_1>0$ with the following property.
Let $R>10,\delta\in [0,\pi/4]$ and let 
$E$ be a $\delta$-twisted ruled band of size $2R$.
Let $\beta$ be a long side of $E$ and 
let $y\in E$ be a point of distance $t$ to the seam
of $E$; then 
\[d_{\cal V}(T_{\Pi_{H(E)}(y)}H(E),V(\beta,E))
\leq C_1e^{t-R}.\]
\end{lemma}
\begin{proof} Let 
$\alpha:[-r,r]\times [-\ell,\ell]\to \tilde M$ be the
standard parametrization of the 
$\delta$-twisted ruled band $E$.
Then $\alpha(0,\ell)$ is contained in the central geodesic
$\gamma\subset H(E)$ of $E$.
By the definition of a
$\delta$-twisted ruled band of size $2R$, for $s\in [-r,r]$ we have
\begin{equation}\label{deviation}
d(\alpha(s,\ell),\Pi_{H(E)}(\alpha(s,\ell)))\leq  
C_2\delta \vert s\vert \end{equation}
where $C_2>0$ is  a universal constant.
Namely, the central geodesic $\gamma$ of $E$ is contained
in the hyperbolic plane $H(E)$. 
The point $\alpha(s,\ell)$ 
can be obtained from the endpoint 
$\alpha(0,\ell)$ of $\gamma$
by a geodesic of length $\vert s\vert \leq r$ which makes an angle
$\delta/2$ to the tangent plane of $H(E)$.

Let $\beta_s\subset H(E)$ be the geodesic connecting
$\alpha(s,0)$ to $\Pi_{H(E)}(\alpha(s,\ell))$. We assume that
$\beta_s$ is 
parametrized proportional to arc length on $[0,\ell]$.
Since the curvature of $\tilde M$ is bounded from 
above by $-1$, comparison shows that for $0\leq t\leq \ell$  
we have
\begin{equation}\label{deviation2}
d(\alpha(s,t),\beta_s(t))\leq C_3 \delta \vert s\vert e^{t-R}
\end{equation}
for a universal constant $C_3>0$.

Parallel transport of tangent planes along geodesics
in $\tilde M$ 
defines horizontal geodesics in the bundle ${\cal V}$.
From the estimate (\ref{deviation2}) we therefore obtain that 
for $t\in [0,\ell]$ we have
\begin{equation}\label{deviation3}
d_{\cal V}(V(\alpha_r,E)(\alpha(r,t)), 
T_{\beta_r(t)}H(E))\leq C_4\delta e^{t-R}.\end{equation}

The geodesics $\beta_{-r},\beta_r$ and the seam 
$\nu_0$ of $E$ 
define three sides of a geodesic quadrangle $Q$ in 
the hyperbolic plane $H(E)$. The length of the
sides $\beta_{-r},\beta_r$ equals 
$\ell$ up to an error of size at most $\delta$.
Since the projection $\Pi_{H(E)}$ is distance
non-increasing, the
length of the side of $Q$ opposite to $\nu_0$ is at most $2r$.
Thus inequality (\ref{deviation}) implies that
if $y\in E$ is at distance $t$ from the seam,
then there is some $z\in \beta_r$ such that
$d(z,\Pi_{H(E)}(y))\leq C_5 e^{t-R}$. 

As the map $z\in H(E)\to T_zH(E)$ is an isometric
embedding of $H(E)$ into ${\cal V}$, we then have 
\[d_{\cal V}(T_zH(E),T_{\Pi_{H(E)}(y)}H(E))\leq C_5e^{t-R}\]
as well.
Together with the estimate (\ref{deviation3}), 
this shows the lemma.
\end{proof}

\begin{remark}\label{higherrank1}
Lemma \ref{planedist} immediately extends to
symmetric spaces $X$ of higher rank as follows.
Let $H\subset X$ be a totally geodesic 
embedded plane of constant curvature $-1$. 
For $\delta >0$ define a $\delta$-twisted 
ruled band in $X$ by rotating the
small sides of an embedded band $\hat E$ in $H$ about the 
central geodesic of $\hat E$ by an angle $\delta$
as described above. Then 
\[d_{\cal V}(T_{\Pi_{H(E)}(y)}H(E),V(\beta,E))\leq
\max\{\delta,C_1e^{t-R}\}.\] 
\end{remark}

The next lemma compares distances in twisted 
ruled bands $E$ with the distance of their projections 
to $H(E)$. To this end denote 
for a $\delta$-twisted ruled band $E$
by $d_E$ the intrinsic path metric on $E$. 
Note that $E$ is a smoothly embedded submanifold of 
$\tilde M$, in particular its tangent plane is defined 
everywhere.

\begin{lemma}\label{projection}
For every $\epsilon >0$ there is a number 
$\delta_1=\delta_1(\epsilon)>0$ with the following
property. Let $R>10$,  
$\delta \leq \delta_1$ and let $E$ be a 
$\delta$-twisted ruled band of size $2R$.
Then for all 
$x,y\in E$ we have 
\[d(\Pi_{H(E)}(x),\Pi_{H(E)}(y))\geq
d_E(x,y)(1+\epsilon)^{-1}.\] 
\end{lemma}
\begin{proof}
Let $\epsilon >0$, let $\delta>0$ and let $E$ be a $\delta$-twisted ruled band
of size $2R>10$.
By the discussion in the proof of 
Lemma \ref{planedist} (or by a standard compactness argument), 
for sufficiently small $\delta$ the distance in ${\cal V}$ between
a tangent plane $T_yE$ of $E$ and the tangent bundle
${\cal T}H(E)$ of $H(E)$ is at most $\epsilon$.

As a consequence, for sufficiently small
$\delta$ the restriction of the projection
$\Pi_{H(E)}$ to $E$ is 
a diffeomorphism onto its image which moreover
is bilipschitz with bilipschitz constant at most $1+\epsilon$. 
\end{proof}

In the following definition, 
the oriented distance of two points on the boundary of 
an oriented surface is taken with respect to the induced 
boundary orientation. The definition is a variant of 
a definition from \cite{KM12}.

\begin{definition}\label{wellattached5}
For numbers $\sigma_1,\sigma_2 \in [0,1/4]$, 
two oriented twisted ruled bands $E_1,E_2$ are called 
\emph{$(\sigma_1,\sigma_2)$-well attached} 
along a common boundary geodesic
$\beta$ if the following holds.
\begin{enumerate}
\item[$\bullet$]
The orientations of $\beta$ induced by the orientations of
$E_1,E_2$ are opposite.
\item[$\bullet$] Let $x_i\in \beta$ be the endpoint on 
$\beta$ of the seam $\nu_i$ of $E_i$ $(i=1,2)$.   
The oriented 
distance along $\beta$ between $x_1,x_2$ 
is contained in the interval $[1-\sigma_1,1+\sigma_1]$.
\item[$\bullet$] Let $v_i\in T_{x_i}^1\tilde M$ be the oriented
tangent of $\nu_i$ at $x_i$ $(i=1,2)$; then 
the angle between $v_2$ and the parallel transport of 
$-v_1$ along $\beta$ is at most $\sigma_2$. 
\end{enumerate}
\end{definition}

Note that in view of Lemma \ref{longarcs}, the first
two properties control the intrinsic geometry
of the attached bands. The third property is used to relate
the intrinsic geometry of the attached bands 
to the extrinsic geometry of the ambient manifold.

Fix a number $b>1$. 
For numbers $m>10,\delta >0$ define an
\emph{$(R,\delta)$-admissible chain} 
of twisted ruled bands to 
be a surface of the form 
$E=E_1\cup\dots\cup E_m$ where $E_i$ is an oriented  
$\delta_i<\delta$-twisted band of size $2R_i$ 
for some $R_i\in [R-\delta,R+\delta]$ 
and where
$E_i$ is $(\delta,R^{-b})$-well attached to $E_{i-1}$ 
along a boundary geodesic which 
is disjoint from $E_{i-2}$.

In what follows, whenever we estimate distances in 
a fibre bundle over $\tilde M$, then these distances
are taken with respect to the natural 
Riemannian metric on the bundle which is induced from the 
Riemannian metric on $\tilde M$.
The distance in $\tilde M$ will simply be denoted by $d$.

$\delta$-twisted ruled bands with their intrinsic metric
are isometrically immersed smooth submanifolds
of $\tilde M$ which are $C^2$-close to 
totally geodesic embedded hyperbolic planes. Thus
the intrinsic curvature is close to $-1$, and by
making $\delta$ smaller we may assume that 
this metric is ${\rm CAT}(-1/2)$.
A rescaled version of 
Lemma \ref{localcat} then shows that the intrinsic 
path metric on any $(R,\delta)$-admissible chain of ruled bands
is locally ${\rm CAT}(-1/2)$. In particular, any two 
points are connected by a unique geodesic.

Our next goal is to 
control the geometry of $(R,\delta)$-admissible chains 
$E_1\cup\dots \cup E_m$ of 
twisted ruled bands by comparing distances for the 
intrinsic path metric with distances in $\tilde M$.

The constant $C_0>0$ in the formulation of the following
lemma is the constant from Lemma \ref{realhyp}.
Up to changing $\delta_2$, the number 
$\min\{1/4,C_0/2\}$ can be replaced
by any other positive constant.

\begin{lemma}\label{anglecontrol}
For every $\epsilon >0$ there are numbers 
$\delta_2=\delta_2(\epsilon)>0$, $R_2=R_2(\epsilon)>10$ 
with the following property. Let $R>R_2$ and 
let $E=E_1\cup\dots\cup E_m$ be an $(R,\delta_2)$-admissible chain of
ruled bands. For $i,j\leq m$      
let $y\in \partial E_i,z\in \partial E_j$ be points 
which are connected by a geodesic $\zeta$ in $E$ 
for the intrinsic metric of length $\ell(\zeta)\leq 
\min\{1/4,C_0/2\}$. Assume that 
$\zeta$ does not meet a short side of any
band in the chain. 
Then 
the length of $\zeta$ 
is at most $(1+\epsilon)d(y,z)$.
\end{lemma}
\begin{proof} For a number $R>10$, a number $\delta\in [0,1/10]$ 
and some $m\geq 1$ let 
$E=E_1\cup \dots \cup E_m$ be an
$(R,\delta)$-admissible chain.

Let $1\leq i<j\leq m$ and let $x\in \partial E_i,y\in \partial E_j$.
Assume that the geodesic $\zeta$ on $E$ 
connecting $x$ to $y$
does not intersect a small side of any of the bands
which make up $E$ and that its length
does not exceed $\min\{1/4,C_0/2\}$.

For $k\leq m$ let 
\[\beta_k=E_{k-1}\cap E_k.\] 
For $i<\ell <j$ let $u_\ell=\zeta\cap \beta_\ell$.
Let $\tau_\ell$ be the distance between $u_\ell$ and
the seam  of the band $E_\ell$. By assumption, for every
$z\in \zeta\cap E_\ell$ the distance between
$z$ and the seam of $E_\ell$ is contained in the
interval $[\tau_\ell-1/4,\tau_\ell +1/4]$. 

By Lemma \ref{localcat},
the intrisinc path metric on $E$ is locally ${\rm CAT}(-1/2)$.
Thus we can use Remark \ref{onlylam} and deduce from
Lemma \ref{longarcs} and its proof   
that up to subdividing $\zeta$ into two
disjoint segments and reversing the orientation of 
one of these segments as well as 
reversing the numbering of the bands in 
the chain, we may assume that 
$\tau_{\ell+1}> \tau_\ell$ for  all $\ell$.
Moreover, we have $j-i\leq 4R$.
  
We now claim that there is a number 
$\chi_0>0$ with the following property. 
For every $i<k < j$ and every $z_k\in \zeta\cap  E_k$ we have
\begin{equation}\label{sumangle}
d_{\cal V}(T_{\Pi_{H(E_k)}(z_k)}H(E_k),T_{\Pi_{H(E_i)}(z_k)}H(E_i))\leq 
\chi_0\sum_{\ell=i}^k\max\{e^{\tau_\ell-R},R^{-b}\}.
\end{equation}
This estimate holds true for every $\delta\in [0,1/10]$.

We proceed by induction on $k-i$.
The claim for $k=i$ is trivial- in fact, 
every number $\chi_0>0$ will do.
Thus assume that the statement
holds true for $k-i<n\leq j-i$ where $n\geq 1$.

Lemma \ref{planedist} shows that 
there is some $u\in \beta_k$ such that 
\begin{equation}\label{angle5}
d_{\cal V}(T_{\Pi_{H(E_k)}(z_k)}H(E_k),
V(\beta_k,E_k)(u))\leq C_1 e^{\tau_k-R}.
\end{equation}

Now $V(\beta_k,E_k)(u)$ is obtained from
the span of $\beta_k^\prime$ and the tangent of the seam
$\nu_k$ of $E_k$ by parallel transport along $\beta_k$.
Similarly, $V(\beta_k,E_{k-1})(u)$ is 
obtained from the span of $\beta_k^\prime$ and 
the tangent of the seam $\nu_{k-1}$
by parallel transport
along $\beta_k$. In particular, 
by the definition of well attached bands, 
\begin{equation}\label{angle4}
d_{\cal V}(V(\beta_k,E_{k-1})(u), V(\beta_k,E_k)(u))\leq 
R^{-b}.\end{equation}
Thus from the estimate (\ref{angle5}) we conclude that
\begin{equation}\label{angle7}
d_{\cal V}(T_{\Pi_{H(E_k)}(z_k)}H(E_k),
V(\beta_k,E_{k-1})(u))\leq C_2\max\{e^{\tau_k-R}, R^{-b}\}.
\end{equation}

Lemma \ref{projection} implies that up to replacing
$C_1$ by $2C_1$, 
the intrinsic distance in $E$ between $z_k$ and $u$
is at most $C_1e^{\tau_k-R}$. Since the intrinsic path
metric on $E$ is ${\rm Cat}(-1/2)$ and since 
the projections $\Pi_{H(E_i)}$ are distance
non-increasing, the estimate (\ref{angle7})
yields that for the proof of 
inequality (\ref{sumangle}), it suffices to show
that 
\[d_{\cal V}(V(\beta_k,E_{k-1}(u)),T_{\Pi_{H(E_i)(u)}}H(E_i))
\leq \chi_0\sum_{\ell=i}^{k-1}\max\{e^{\tau_\ell-R},R^{-b}\}.\]

Lemma \ref{planedist} allows to replace
$V(\beta_k,E_{k-1}(u))$ by 
$T_{\Pi_{H(E_{k-1})(u)}}H(E_{k-1})$. 
The estimate (\ref{sumangle}) now follows from
the induction hypothesis provided that
the constant $\chi_0>0$ is sufficiently large
(in particular, it has to be chosen larger than $2C_2$).

For $\rho>0$ there is a number $r(\rho)>0$ with the 
following property. Let $H_1,H_2\subset\tilde M$ be two 
totally geodesic real hyperbolic planes. Assume that
$x\in H_1$ and that $d_{\cal V}(T_xH_1,T_{\Pi_{H_2}(x)}H_2)<r(\rho)$; then
the restriction of the projection $\Pi_{H_2}$ to 
the ball of radius one about $x$ in $H_1$ is a 
$(1+\rho)$-bilipschitz diffeomorphism onto its image. 
Moreover, for every $y\in H_1$ with $d(x,y)\leq 1$ we have
$d(y,\Pi_{H_2}(y))\leq \rho$.

Let $\epsilon >0$. By inequality (\ref{sumangle}),
by Lemma \ref{projection}, Lemma \ref{realhyp} and 
by Lemma \ref{longarcs} and its proof,
there are numbers $p>0$, $R_2>0$ with the following property.
Let $E$ be an $(R,\delta)$-admissible chain for some
$R\geq R_2$ and some $\delta<1/10$.
Let $\zeta:[0,a]\to E$ be any geodesic arc 
of length $a\leq \min\{1/4,C_0/2\}$ 
as above which 
connects a point $x\in \partial E_i$ to a point
$y\in \partial E_j$. Assume that $j-i\geq p$. 
Then there are numbers
$i_0\in [i,i+p],j_0\in [j-p,j]$ with the following property.

Let $0\leq s\leq t\leq a$ be such that
$\zeta(s)\in \beta_{i_0},\zeta(t)\in \beta_{j_0}$; then 
\begin{equation}\label{sumangle2} 
d_{\cal V}(T_{\Pi_{H(E_{i_0})}(\zeta(s))}H(E_{i_0}),
T_{\Pi_{H(E_{j_0})}(\zeta(s))}H(E_{j_0}))\leq r(\epsilon/2)/3.
\end{equation}
As before, this estimate holds true for all 
$\delta\in [0,1/10]$.

With this number $p>0$,
it follows from the estimate (\ref{sumangle}), 
Lemma \ref{projection} and its proof 
and the definition of an admissible chain that there
is a number $\sigma=\sigma(\epsilon,p)<1/10$ 
with the following property.

Let $R>R_2$ and let 
$E=E_1\cup\dots \cup E_p$ be any $(R,\sigma)$-admissible chain 
of twisted ruled bands. Let 
$x\in E$; then for $1\leq i\leq j\leq p$ we have
\begin{equation}\label{finitelength}
d_{\cal V}(T_{\Pi_{H(E_i)}(x)}H(E_i),T_{\Pi_{H(E_j)}(x)}H(E_j))\leq r(\epsilon/2)/3.
\end{equation}
The point is here that the number $p$ is fixed, and 
that the number $\sigma$
can be chosen arbitrarily small.


%

Now let $\delta_1(\epsilon)>0$ be as in Lemma \ref{projection} and let 
$\delta_2=\min\{\sigma,\delta_1(\epsilon)\}$.
Let $R>R-2$, 
$\delta<\delta_2$, let 
$m>0$ be arbitrary and let 
$E=E_1\cup \dots \cup E_m$ be an $(R,\delta)$-admissible chain.
Let $1\leq i<j\leq m$ and let $y\in E_i,z\in E_j$  
be such that $d(y,z)\leq \min\{1/4,C_0/2\}$
as before and that moreover the geodesic in $E$ connecting
$x$ to $y$ does not meet a small side of a band in the chain.
Choose $i\leq i_0<j_0\leq j$ 
with $i_0-i\leq p,j-j_0\leq p$ as above. 

Let $\Pi=\Pi_{H_{i_0}}$. 
By the choice of $\delta$,  
the restriction of the projection $\Pi$ to the ball of radius one
about $y$ in $E$ is injective. 
Thus the geodesic $\eta$ in $H(E_{i_0})$ connecting 
$\Pi(y)$ to $\Pi(z)$ crosses through the lines
$\hat \beta_\ell=\Pi(\beta_\ell)$ where $i<\ell <j$
in increasing order.

Assume that $\eta$ is parametrized 
by arc length on an interval $[0,c]$.
Let $t_\ell\geq 0$ be such that 
$\eta(t_\ell)=\eta\cap \hat \beta_\ell$. 
Let $\zeta_\ell\in E$ be the preimage of $\eta(t_\ell)$ under the 
map $\Pi\vert E$. Using again the choice of $\delta$,
the length $q_\ell$ 
of a shortest geodesic in $H(E_\ell)$ connecting
$\Pi_{H(E_\ell)}(\zeta_{\ell -1})$ to $\Pi_{H(E_\ell)}(\zeta_\ell)$
does not exceed $(1+\epsilon)(t_\ell -t_{\ell -1})$.
As $\delta<\delta_1(\epsilon)$, 
Lemma \ref{projection} shows that 
the length of 
a shortest geodesic in $E_\ell$ connecting
$\zeta_{\ell-1}$ to $\zeta_\ell$ is not 
bigger than $q_\ell(1+\epsilon)$.
Summing over $\ell$ implies the lemma.
\end{proof}

\begin{remark}\label{higherrank2}
Lemma \ref{anglecontrol} is valid without change
for chains of $\delta$-twisted ruled bands 
of size $R$ in 
higher rank symmetric spaces provided that these
bands are constructed as in Remark \ref{higherrank1} and
are $(\delta,e^{-\kappa R})$-well attached for some
$\kappa >0$.  
\end{remark}

The thin parts of $(R,\delta)$-skew pants are not $\delta$-twisted
ruled bands, but they are exponentially close to such bands in 
a sense we now specify.
Namely, define an \emph{approximate $\delta$-twisted ruled band
of size $R$} to be a ruled surface $E$ with the following
property. 

Fix a number $\kappa\in (0,1)$ whose precise value will be determined later.
We require that 
there is a $\delta$-twisted ruled band $E^0$ 
of size $R$, with short sides $\nu^0_{-R},\nu^0_R$, and there are
geodesics $\nu_{-R},\nu_R$ parametrized proportional to arc
length on $[-r,r]$ with 
\[d(\nu_{i}(-r),\nu_i^0(-r))\leq e^{-\kappa R},
d(\nu_i(r),\nu_i^0(r))\leq e^{-\kappa R}\, (i=-R,R)\] 
such that
$E$ is obtained by connecting  
for each $s\in [-r,r]$ the points 
$\nu_{-R}(s)$ and $\nu_R(s)$
by a geodesic. 
We call the geodesics $\alpha_{-r},\alpha_r$ 
connecting the endpoints of $\nu_{-R},\nu_R$ 
the \emph{long sides}
of the ruled band $E$, and we call the 
shortest geodesic in $\tilde M$ connecting 
the two long sides of $E$ 
the \emph{seam} of the band.

Let $\alpha_0:[-r,r]\times [-\ell,\ell]\to E_0$ be the
standard parametrization of the twisted ruled band
$E_0$ as described in the beginning of this section 
and let $\alpha:[-r,r]\to [-\ell,\ell]\to E$ be the
parametrization of $E$ defined by requiring that
$t\to \alpha(s,t)$ is the geodesic connecting
$\nu_{-R}(s)$ to $\nu_R(s)$ parametrized proportional
to arc length on $[-R,R]$.
Hyperbolicity and comparison shows that 
\[d(\alpha_0(s,t),\alpha(s,t))\leq e^{-\kappa  R}\] for all $s,t$.

The notion of an $(R,\delta)$-admissible chain 
is also defined for approximate $\delta$-twisted ruled bands.
As in Lemma \ref{anglecontrol} we conclude

\begin{corollary}\label{anglecontrol2}
For every $\epsilon >0$ 
there is a number $\delta_3=\delta_3(\epsilon)>0$
with the following property. Let $m>10$ and let
$E=E_1\cup\dots\cup E_m$ be an 
$(m,\delta_3)$-admissible chain of 
approximate ruled bands. Let
$x,y$ be points which 
are connected by a geodesic $\zeta$ in $E$ for the intrinsic metric.
Assume that $\zeta$ does not meet the short side of any
band in the chain. Then 
the length of $\zeta$ is at most 
$(1+\epsilon)d(x,y)$. 
\end{corollary}

\section{Surfaces glued from skew pants}\label{skewpants}

In this section we glue skew-pants
to surfaces and investigate their geometry.
We continue to 
use the assumptions and notations from Section \ref{twisted} and 
Section \ref{constructing}.

Let $P\subset M$ be a skew pants defined
by an $(R,\delta)$-well connected pair of framed  
tripods $x=((v_1,v_2,v_3),E)$, 
$y=((w_1,w_2,w_3),F)$ with footpoints
$p,q$. For the remainder of the section, only the 
real planes defined by the tripods 
$(v_1,v_2,v_3)$ and 
$(w_1,w_2,w_3)$ are relevant, so we drop the information on the
frames $E,F$.

From the $(R,\delta)$-well connected
tripods $x,y$ we construct a ruled surface in $M$ 
in the homotopy class of $P$ as follows.

Let $\alpha$ be a boundary geodesic of $P$.
It contains in its $\kappa_4 e^{-R}$-neighborhood 
a long side of each of 
the immersed hexagons $H_R(v_1,v_2,v_3)$ and $H_R(w_1,w_2,w_3)$. 
Here $\kappa_4>0$ is as in Section \ref{constructing}.
In particular, there is a geodesic arc $\xi_{p,\alpha}:[0,s]\to M$ 
connecting $p=\xi_{p,\alpha}(0)$ 
to a point $\xi_{p,\alpha}(s)$ on $\alpha$ which
meets $\alpha$ orthogonally at 
$\xi_{p,\alpha}(s)$ 
and which is determined by the homotopy type of $P$ as follows. 
Lift the hexagon $H_R(v_1,v_2,v_3)$ 
locally isometrically to a totally geodesic embedded 
hexagon $\tilde H$ in $\tilde M$, and lift
$\alpha$ to a geodesic line $\tilde \alpha$ 
whose $\kappa_4e^{-R}$-neighborhood contains a long
side of $\tilde H$. 
Let $\tilde\xi_{p,\alpha}$ be the shortest geodesic
connecting the center of the hexagon $\tilde H$ to 
$\tilde \alpha$ and let $\xi_{p,\alpha}$ be the 
projection of $\tilde \xi_{p,\alpha}$ to $M$.



Let $(\xi_{p,\alpha},\xi_{p,\beta},\xi_{p,\gamma})$ be the triple of 
these geodesic arcs in $M$ 
connecting the footpoint $p$ of the tripod 
$(v_1,v_2,v_3)$ to the three boundary geodesics 
$\alpha,\beta,\gamma$ of the
skew pants.
The geodesic arcs in $M$ which are homotopic to 
$\xi_{p,i}\circ \xi_{p,j}^{-1}$ with fixed endpoints 
$(i\not=j\in \{\alpha,\beta,\gamma\})$ define 
a geodesic triangle $\partial \Sigma$. Since the $\kappa_4e^{-R}$-neighborhoods
of the boundary geodesics of the skew-pants $P$ 
contain the long sides of the hexagon $H_R(v_1,v_2,v_3)$, 
by convexity the side lengths of $\partial \Sigma$ 
are within $2\kappa_4e^{-R}$ of the 
number $2r$ used in the definition of a $\delta$-twisted band of size $R$
in Section \ref{twisted}.

Choose a vertex $z$ of $\partial \Sigma$ 
and connect this vertex to each point on the opposite
side by a geodesic arc whose homotopy class is 
determined by the homotopy classes of the 
two sides of $\partial \Sigma$
which are incident on $z$.
This defines a ruled surface $\Sigma\subset M$
with boundary $\partial \Sigma$ 
(which however depends on the
choice of a vertex of $\partial \Sigma$). 
We call such a ruled
surface a \emph{center triangle} for the skew-pants. 
Thus each pair of well connected framed tripods defines
a skew-pants together with the choice of two 
center triangles. By convexity and the fact that the
hexagon $H_R(v_1,v_2,v_3)$ is totally geodesic,
such a ruled triangle is contained in the
$\kappa_4e^{-R}$-neighborhood of $H_R(v_1,v_2,v_3)$.

Up to modifying the skew-pants $P$ by a homotopy with fixed boundary,
we may assume that the center triangles are embedded in $P$.
Then the complement of these center triangles in 
$P$ consists of three rectangles. The boundary of 
each such rectangle
is composed of four sides which are geodesic segments in $M$. 
Two sides are sides of a
center triangle, the other two sides are geodesic
subarcs of the boundary geodesics of $P$.
We call the sides contained in the center triangle
the \emph{short sides} of the rectangle, the other two
sides are called the \emph{long sides}.

Parametrize the short sides of such a rectangle $Q$ 
proportional to 
arc length on $[-r,r]$. Use this parametrization to
construct a ruled surface with boundary $Q$ and 
with ruling 
containing the long sides
in the boundary of $P$. 
By construction, there is a number $\nu<\kappa_4\delta$
such that this ruled surface is an 
approximate $\nu$-twisted ruled band
of size $6R-2\tau+\chi$ for a number
$\chi\in [-\delta,\delta]$, where $\tau>0$ is the distance
of the center of an equilateral triangle of side length
$2\tau$ in the hyperbolic plane to each of its sides.
The long
sides of these bands are subsegments of the boundary geodesics
of $P$. 

The three ruled bands are glued to the 
center triangles along the short sides of their boundary. 
The union of the
three ruled bands and the two center triangles 
defines a piecewise ruled surface which is a pair of pants
with geodesic boundary.
We call such a piecewise ruled surface 
a \emph{$(R,\delta)$-geometric skew-pants}, or
simply a geometric skew pants if we do not have to 
specify the size parameters $(R,\delta)$.

A geometric skew pants $P$ is a pair of pants with 
a piecewise smooth
Riemannian metric with geodesic boundary.
This piecewise smooth metric 
defines a path metric on $P$. 
Lemma \ref{localcat} 
shows that this path
metric is locally ${\rm CAT}(-1)$.
Note that an $(R,\delta)$-skew pants $P$ 
can be equipped with a structure of 
a geometric $(R,\delta)$-skew pants, but such a structure
is not unique.
 
An $(R,\delta)$-skew pants $P$ has three boundary geodesics.
Each pair of such geodesics is connected by a shortest
geodesic arc in the homotopy class defined by the 
skew-pants. These geodesic
arcs are called the \emph{seams} of $P$.

Each boundary geodesic contains an endpoint of 
precisely two seams, and
these endpoints decompose the boundary geodesic into
two subarcs of roughly the same length. By Lemma \ref{midpoint}
and comparison, 
the angle between the direction of a seam at a point in 
a boundary geodesic $\gamma$ and a
direction in the $\mathbb{K}$-orthogonal complement of 
$\gamma^\prime$ is exponentially small in $R$.

Recall that by convention, a skew-pants is oriented and hence
each of its boundary geodesics is oriented as well.
Following Section \ref{twisted}, we can now define

\begin{definition}\label{wellattachedpants}
For a number $\sigma>0$, two skew-pants $P,P^\prime$
are \emph{$\sigma$-well attached} along a common boundary geodesic
$\beta$ if the following holds true.
\begin{enumerate}
\item[$\bullet$] The orientations of $\beta$ as a boundary geodesic
of $P$ and $P^\prime$ are opposite.
\item[$\bullet$] Let $x$ be an endpoint of a seam of $P$ on
$\beta$. Then there is an endpoint $y$ of a seam of 
$P^\prime$ on $\beta$ whose oriented distance to 
$x$ is contained in the interval $[1-\sigma,1+\sigma]$.
\item[$\bullet$] Let $v_1$ be the direction of the
seam of $P$ at $x$ and let $v_2$ be the direction of the
seam of $P^\prime$ at $y$; then the angle
between $v_2$ and the image of $-v_1$ under parallel transport
along the oriented subarc of $\beta$ connecting $x$ to $y$ is 
at most $\sigma$.  
\end{enumerate}
\end{definition}

As before, let $b>1$ be a fixed number. Recall from Section \ref{twisted}
the definition of an admissible chain for this number $b>1$.
We are now ready to show

\begin{proposition}\label{incompressible}
There are numbers $\delta_4 \in (0,\pi/4],R_4>10$ 
with the following property.
Let $R>R_4$ and let 
$S\subset M$ be a piecewise immersed closed surface composed
of finitely many $(R,\delta_4)$-skew pants which are 
$R^{-b}$-well attached along their common boundary geodesics.
Then $S$ is incompressible.
\end{proposition}
\begin{proof} Let $\delta \in (0,\pi/4]$, let $R>10$ and let
$S\subset M$ be an immersed surface which is 
composed of finitely many $(R,\delta)$-skew pants. 
Equip each of these skew-pants 
with a structure of a geometric $(R,\delta)$-skew pants.
By construction and Lemma \ref{localcat}, the length metric on $S$
defined by the piecewise ruled pairs of pants 
is locally ${\rm CAT}(-1)$.
Thus this metric lifts to a ${\rm CAT}(-1)$-metric on 
the universal covering $\tilde S$ of $S$.

The $(R,\delta)$-skew pants define a geodesic pants decomposition
${\cal P}$ of $S$. Let $\sigma >0$.
By Lemma \ref{projection} 
and Lemma \ref{realhyp},
for sufficiently large $R$ and sufficiently small $\delta$,
say for all $\delta \leq \hat \delta$, 
this pants decomposition is $(R^\prime,\sigma)$-tight 
for some $R^\prime>0$ (here the constant $R^\prime$ is determined
by $R$ and the definition of $(R,\delta)$-skew pants)
and centrally $C_0/2$-thick where
$C_0>0$ is as in Lemma \ref{realhyp}.

For $\rho>0$ define the \emph{$\rho$-thick part} of a skew-pants $P$ to be
the set of all points $x$ so that the open metric ball of radius $\rho$ about
$x$ with respect to the intrinsic path metric does not intersect 
the boundary of the skew-pants. 
By Lemma \ref{longarcs}, there is a number
$\rho_0>0$ not depending on $R,\delta$ 
so that any geodesic arc $\zeta$ on $S$ which is not 
contained in an admissible chain of twisted bands meets
the $\rho_0$-thick part of some skew-pants (here as before, we assume that
$R>10$ is sufficiently large and that $\delta>0$ is sufficiently small).

Now recall that a skew-pants is a union of 5 ruled surfaces with 
geodesic boundary which are close to being totally geodesic immersed in 
$M$. By construction, a geodesic arc $\zeta$ which intersects the 
$\rho_0$-thick part of a skew-pants crosses through at most two of the
boundary arcs of such a surface. As the angle with which two of these ruled
surfaces meet at a common boundary geodesic tends to zero with 
$\delta$, we conclude that for a given fixed $\epsilon >0$ 
the following holds true. Let $\zeta$ be a geodesic arc of length
$\rho_0$ on $S$ which intersects the $\rho_0$-thick part of a 
skew-pants. Then for 
the length of the lift of $\zeta$ 
to $\tilde M$ is at most
$(1+\epsilon)$-times the distance in $\tilde M$ between its endpoints
provided that the number $\delta >0$
used in the construction is sufficientlly small.

Thus by Proposition \ref{cannonmap},
Corollary \ref{anglecontrol2} and the definition of a geometric
skew-pants, 
for the proof of the proposition it now suffices to show
that for every $\rho >0$ there is a number
$\epsilon =\epsilon(\rho)>0$ with the following property.
Let $\gamma:\mathbb{R}\to \tilde M$ be a piecewise smooth curve.
Assume that for every subarc $\gamma[t,t+\rho]$ 
of $\gamma$ of length $\rho$ we have 
$d(\gamma(t),\gamma(t+\rho))
\geq \rho/(1+\epsilon)$; then $\gamma$ is 
an $L$-quasi-geodesic in $\tilde M$ for a number
$L>1$ only depending on $\rho$ and $\epsilon$.

However, the existence of such a number $\epsilon>0$ follows from 
hyperbolicity. Namely, for $\rho>0$ let 
$\gamma:\mathbb{R}\to \tilde M$ be a piecewise smooth
curve as in the previous paragraph, and let 
$\hat\gamma$ be the piecewise geodesic 
in $\tilde M$ such that for all $m\in \mathbb{Z}$ we have
\begin{enumerate}
\item[$\bullet$]  
$\hat\gamma(m\rho/2)=\gamma(m\rho/2)$ and
\item[$\bullet$]
$\hat\gamma[m\rho/2,(m+1)\rho/2]$ is a geodesic
parametrized proportional to arc length.  
\end{enumerate}

For each $m$ let $\alpha(m)$ be the breaking
angle of the segments of $\hat \gamma$ which come together at
$\hat \gamma(m\rho/2)$.
By this we mean that $\pi-\alpha(m)$ is the angle at
$\hat\gamma(m\rho/2)$ of the triangle in $\tilde M$ 
with vertices
$\hat\gamma((m-1)\rho/2),\hat\gamma(m\rho/2),\hat\gamma((m+1)\rho/2)$).

By the assumption on $\gamma$, the lengths of 
the sides adjacent to $\hat\gamma(m\rho/2)$
of this triangle 
are contained in the interval
$[\rho/2(1+\epsilon),\rho]$, and 
the length of the opposite side is at least
$(1+\epsilon)^{-1}$ times the sum of the lengths of the adjacent
sides. By angle comparison, for any number $\delta>0$ there
is some $\epsilon=\epsilon(\delta)<1/2$ such that 
for this $\epsilon$, the breaking angles
$\alpha(m)$ do not exceed $\delta $. 

On the other hand, by hyperbolicity, there
is a number $\delta=\delta(\rho)>0$ and a number $L>1$ 
with the following property.
Let $\gamma:\mathbb{R}\to \tilde M$ be a piecewise geodesic
composed of geodesic segments of length at least 
$\rho/4$. 
If the breaking angles of $\gamma$ at the breakpoints do not exceed 
$\delta$ then $\gamma$ is an $L$-quasi-geodesic.

Together 
this shows the existence of a number
$\epsilon=\epsilon(\delta(\rho))$ as required above and completes the 
proof of the proposition.
\end{proof}

\section{The glueing  equation}\label{theglueing}

In this section we show that for $G\not=SO(2m,1)$ 
$(m\geq 1)$ it is possible to construct a closed 
immersed surface $S$ in $M=\Gamma\backslash G/K$ which
is composed of $(R,\delta_4)$-skew-pants 
for some $R>R_4$ in such a way that 
the assumptions in Proposition \ref{incompressible}
are satisfied.
Proposition \ref{incompressible} then implies that the 
surface $S$ is incompressible in $M$.

It is only in this section that we fully use the assumption that 
$M=\Gamma\backslash G/ K$ for a simple rank one 
Lie group $G$ and a cocompact torsion free lattice 
$\Gamma<G$.
First we use controlled rate of mixing for the 
frame flow on $M$ to construct sufficiently well distributed
$(R,\delta)$-skew pants with the method from
Lemma \ref{framedconst}. To specify the idea of good
distribution of these pants we equip them with a weight
function constructed from the Lebesgue measure on a
suitably chosen bundle over
the universal covering $\tilde M$  of $M$. 
The fact that this measure is invariant under the action of 
the entire group $G$
is essential for the argument.
We also use a property which only holds for 
simple rank one Lie-groups of 
non-compact type different from $SO(2m,1)$.
Namely, let $K_0<G$ be the 
compact stabilizer of a unit vector in $T^1\tilde M$
(which is a compact subgroup of the special orthogonal
group $SO(\ell-1)$ where $\ell>0$ is the real dimenison of $M$).
Then the 
component of the identity of the
centralizer of $A$ in $K_0$ 
contains 
$-{\rm Id}\in K_0<SO(\ell-1)$. This property does not
hold for $SO(2m,1)$. We refer to \cite{B12} for 
comments prior to this work why this property is useful 
for the construction of incompressible surfaces.

We begin with a general observation about the construction
of incompressible surfaces, not necessarily in locally
symmetric manifolds. To this end
fix a number $\delta<\delta_4$ and a 
number $R>R_4$ as in Proposition \ref{incompressible}.
We allow to decrease $\delta$ and increase $R$ 
throughout the construction. 

Let ${\cal P}(R,\delta)$ be the collection of all oriented
$(R,\delta)$-skew pants in $M$. 
The boundary of 
each such skew pants consists of 
a triple of closed geodesics whose lengths are
contained in the interval
$[8R+2L(R)-\delta,8R+2L(R)+\delta]$ 
(compare Section \ref{constructing}),
with properties as specified in the previous sections. 
Since for every $k>0$ there are only finitely many 
closed geodesics in $M$ of length at most $k$, 
the set ${\cal P}(R,\delta)$ is finite.
If $P$ is a geometric skew pants defining a skew-pants
in ${\cal P}(R,\delta)$ then we write 
$P\in {\cal P}(R,\delta)$ although by definition, 
a skew pants in ${\cal P}(R,\delta)$ is not equipped with
a preferred geometric structure.


For $b>1$ and $\delta<\delta_4$ 
as in Proposition \ref{incompressible} 
define a graph ${\cal G}(R,\delta)$ whose vertex set is the
set ${\cal P}(R,\delta)$ and where two such vertices $P_1,P_2$ are connected
by an edge if the following two 
properties hold true.
\begin{enumerate} 
\item $P_1,P_2$ have precisely one cuff 
$\gamma$ in common.
\item 
$P_1,P_2$ are $R^{-b}$-well attached along $\gamma$.
\end{enumerate}

We label the edge in ${\cal G}(R,\delta)$ 
connecting the vertices $P_1$ and $P_2$ with the 
common cuff $\gamma\subset P_1\cap P_2$ (here $\gamma$
is viewed as an unoriented geodesic).
Note that the requirement (2) above
does not give any restriction on the homotopy class of 
a geodesic arc with endpoints on $\gamma$ which
determines the homotopy classes of 
the two boundary geodesics of $P_i$
distinct from $\gamma$ $(i=1,2)$.


For each vertex $P$ of ${\cal G}(R,\delta)$, the edges of ${\cal G}(R,\delta)$
incident on $P$ are labeled with three distinct labels $1,2,3$ corresponding
to the three distinct cuffs of $P$. Let ${\cal E}_i(P)$ be the set of 
edges with label $i$ 
$(i=1,2,3)$.

\begin{definition}
An \emph{admissible weight function} $f$ assigns a real valued weight
to each edge of ${\cal G}(R,\delta)$. These weights satisfy the 
following \emph{glueing equations}: 
For each vertex $P$ of ${\cal G}(R,\delta)$,
there are three glueing equations
(one being a consequence of the other two) 
\[\sum_{e\in {\cal E}_i(P)}f(e)=\sum_{e\in {\cal E}_{i+1}(P)}f(e).\]
\end{definition}
Here the index $i$ is taken modulo three. 
We call an admissible weight function
a \emph{solution to the glueing equation}.

\begin{lemma}\label{solution1}
If there is a non-negative non-trivial admissible weight function
then there is a non-negative non-trivial integral admissible weight
function.
\end{lemma}
\begin{proof} Since the coefficients of the glueing
equations are integral, 
each gluing equation cuts out a rational hyperplane in the space of 
all weight functions on the set ${\cal E}$ 
of edges of ${\cal G}(R,\delta)$. 
Thus if there is a non-negative non-trivial 
admissible weight function, then there is a non-negative
non-trivial admissible weight function with rational
weights, and such a function can be multiplied with an
integer to yield a non-negative integral admissible weight function.
This shows the lemma.
\end{proof}

\begin{proposition}\label{solution}
Each non-negative non-trivial integral admissible weight function
for ${\cal G}(R,\delta)$ 
defines an incompressible surface in $M$.
\end{proposition}
\begin{proof}  
Let $f$ be a non-negative non-trivial 
integral admissible weight function
for ${\cal G}(R,\delta)$.
Let $v_1,\dots,v_k$ be those vertices of ${\cal G}(R,\delta)$ 
which are adjacent to edges with positive weight. For each such vertex $v$ 
let \[\ell(v)=\sum_{e\in {\cal E}_1(v)}f(e).\]
Choose $\ell(v)$ copies of the skew-pants $P_v$ corresponding to 
$v$. For each cuff $\gamma_i$ of $P_v$ connect these copies
to copies of the skew-pants with the same cuff $\gamma_i$ 
as prescribed by the weight: if the edge $e$ connects $v$ to 
$v^\prime$ then attach $f(e)$ copies of $P_v$ to $f(e)$ copies of 
$P_{v^\prime}$. By the glueing equation, this can be done in 
such a way that each cuff of each of the skew  pants $P_v$ 
is glued to precisely one cuff of a neighboring pants, and
orientations of these skew pants match. 
As the consequence, the union of these skew-pants defines the homotopy class
of a closed oriented surface $S$ in $M$.
By Proposition \ref{incompressible}, this surface
is incompressible.
\end{proof}

We are left with showing the existence of 
a non-negative non-trivial solution to the 
glueing equation. This is the most subtle part of the 
construction, and it is accomplished using ideas 
from \cite{KM12}. 
%

Let $\lambda$ be the normalized  
Lebesgue measure (of volume one) on the
bundle ${\cal F}\to T^1M\to M$ 
of orthonormal $\mathbb{K}$-frames.
Recall that ${\cal F}$   
is an $SO(n-1)$-principal bundle
(or $SU(n-1)$-principal bundle  
or $Sp(n-1)$-principal bundle
or ${\rm Spin}(7)$-principal bundle) 
over the smooth closed manifold $T^1M$. 
The measure $\lambda$ lifts to a $G$-invariant Radon measure
$\tilde \lambda$ on the bundle 
$\tilde {\cal F}\to T^1\tilde M\to \tilde M$ 
of orthonormal $\mathbb{K}$-frames in 
$T\tilde M$. The group $G$ 
acts simply transitively on $\tilde {\cal F}$.

By Lemma \ref{framedconst}, we can construct $(R,\delta)$-skew pants
by connecting framed tripods with arcs obtained from orbit
segments for the frame flow which begin and end uniformly near
the tripods. To make the idea of 
being uniformly near quantitiative we first construct for each
frame $F\in {\cal F}$ a neighborhood in an 
$G$-equivariant way.
To this end 
let now $\delta <\delta_4/2$.  
Choose a point $\tilde z\in \tilde {\cal F}$ and 
a smooth function 
$f_{\tilde z}:\tilde {\cal F}\to [0,\infty)$ 
which is supported in the
$\delta$-neighborhood of $\tilde z$.
We assume that 
\[\int f_{\tilde z}d\tilde\lambda=1.\]

For $\tilde u\in \tilde {\cal F}$ let 
$\psi\in G$ be an isometry 
which maps $\tilde u$ to $\tilde z$ and define 
$f_{\tilde u}=f_{\tilde z}\circ \psi$. 
Via the projection 
\[Q:\tilde {\cal F}\to {\cal F},\] 
the functions $f_{\tilde u}$ 
project to functions
$f_u$ on the frame bundle ${\cal F}\to T^1M\to M$ 
which are defined
as follows. For $u\in {\cal F}$ choose some $\tilde u$ with 
$q(\tilde u)=u$ and put $f_u(v)=\sum_{Q(\tilde v)=v}f_{\tilde u}(\tilde v)$.
Note that this does not depend on any choices made.

Let 
\[{\cal F\cal T}\to M\] be the bundle of framed tripods over
$M$ 
(see Section \ref{constructing} for the definition of a framed tripod).
Recall that for $R>1$ and $i=1,2,3$
a framed tripod $((v_1,v_2,v_3),F)$ 
defines a frame $F_i$ in the fibre of ${\cal F}$ over $\Phi^{R}v_i$. 
For a framed tripod $z=((v_1,v_2,v_3),F)
\in {\cal F\cal T}$ and a number $R>0$ 
define a function $b_{z,R}$ on 
the space ${\cal F}^3$ of triples of points in ${\cal F}$ 
by 
\[b_{z,R}(z_1,z_2,z_3)=\prod_{i=1}^3 f_{(\Phi^{R}v_i,F_i)}(z_i)\]
(here $(\Phi^{R}v_i,F_i)$ is a point in the
bundle ${\cal F}$ whose 
basepoint in $T^1M$ is the vector $\Phi^{R}v_i$).

The involution 
\[{\cal A}:{\cal F}\to {\cal F}\] 
which replaces the 
base vector of the frame and the first vector in the fibre
by its negative preserves the normalized
Lebesgue measure $\lambda$.

Denote by
$\underline{\Psi^t}$ the product frame flow on ${\cal F}^3$.
Then for suffciently large $R>0$, a five-tuple of points
$(x,y,u_1,u_2,u_3)\in {\cal F\cal T}^2\times {\cal F}^3$
which consists of a pair of 
framed tripods $(x,y)\in {\cal F\cal T}^2$ and some 
$(u_1,u_2,u_3)\in {\cal F}^3$ with  
\[b_{x,R}({\cal A}^3\underline{\Psi^R}(u_1,u_2,u_3))b_{y,R}(u_1,u_2,u_3)>0\]
determines an $(R,\delta)$-skew pants. 
Namely, the framed tripod $x$ determines 
the frames $F_1,F_2,F_3$. For each $i$, the frame $u_i$ is contained 
in the $\delta$-neighborhood of $F_i$, and its image under the
map ${\cal A}\Psi^R$ is contained in the
$\delta$-neighborhood of the frame determined by the framed tripod $y$.
Lemma \ref{framedconst} now shows that
$(x,y,u_1,u_2,u_3)$ defines an $(R,\delta)$-skew-pants. 

Our next goal is to observe that the $(R,\delta)$-skew pants constructed
in this way abound. To this end we use exponential mixing 
with respect to 
the Lebesgue measure of the frame flow
on the bundle ${\cal F}$. 
and we let $\lambda^3$ be the product measure on ${\cal F}^3$.
The volume of $\lambda^3$ equals one.

\begin{lemma}\label{mixing2}
There is a number $\kappa>0$ 
such that for any two framed tripods $x,y$ we have
\[\int b_{x,R}({\cal A}^3\underline{\Psi^R} u)
b_{y,R}(u)d\lambda^3(u)\geq 1-e^{-\kappa R}/\kappa.\]
\end{lemma}
\begin{proof} Since the frame flow on ${\cal F}$ is exponentially 
mixing with respect to the Lebesgue measure
(this is equivalent to exponential decay of matrix coefficients,
see \cite{BM00})
and the functions
$f_z$ are fixed, there is a number $\kappa_0>0$ 
such that for all frames $y,z\in {\cal F}$ and all 
$R\geq 0$ we have 
\[\int f_y({\cal A}\Psi^R v)f_z(v)d\lambda(v)\geq 1-e^{-\kappa_0R}/\kappa_0.\]

Taking a triple of frames and multiplying the result shows
the lemma.
\end{proof}

\begin{remark}
Lemma \ref{mixing2} is the only part of the 
argument which uses
controlled decay of correlation for the frame flow on 
${\cal F}$. In fact, it is immediate from our discussion
that polynomial mixing with exponent at least two is sufficient
for the proof of the main theorem from the introduction.
\end{remark}


A tripod $(v_1,v_2,v_3)$ is determined by the unit tangent vector
$v_1$ and 
the oriented normal of $v_1$ in the oriented real plane 
defined by the tripod. Thus there is a natural bundle isomorphism from 
the bundle of framed tripods onto 
the bundle ${\cal F}$.
The symmetry of order three which cyclically
permutes the vectors in the 
tripod induces a symmetry of order three in the bundle ${\cal F}$
which preserves the Lebesgue measure.

Let $x=((v_1,v_2,v_3),F),
y=((w_1,w_2,w_3),E)$ be any two framed tripods.
These framed tripods define two functions
$b_{x,R}$ and $b_{y,R}$ on ${\cal F}^3$.
With the notations from Proposition \ref{incompressible},
let $R>R_4$ be sufficiently large that
$e^{-\kappa R}/\kappa <R^{-b}$.
Let $\mu$ be the measure on ${\cal F\cal T}^2\times {\cal F}^3$
defined by
\begin{align} 
d\mu(x,y,F_1,F_2,F_3)& \notag\\ 
= b_{x,R}(F_1,F_2,F_3)& 
b_{y,R}({\cal A}\Psi^{2R}F_3,{\cal A}\Psi^{2R}F_2,{\cal A}\Psi^{2R}F_1)
d\lambda(x)d\lambda(y)d\lambda^3(F_1,F_2,F_3).\notag\end{align}
By Lemma \ref{mixing2} and Fubini's theorem, 
the total volume of $\mu$ is contained in the 
interval $[1-e^{-\kappa R}/\kappa,1]$.
Moreover, $\mu$ is invariant under the natural action of the cyclic
group $\Lambda$ of order three which acts as a group of rotations
on the well connected tripods and as a cyclic group of permutations
on the frames. By the above discussion, 
every point $z\in {\rm supp}(\mu)$ determines a 
geometric skew-pants $P(z)$. 
Forgetting the geometric structure of $P(z)$ determines 
a natural map
\[\hat P:{\rm supp}(\mu)\to {\cal P}(R,\delta).\]

Let 
\[{\cal S}\to T^1M\] 
be the bundle over $T^1M$ whose
fibre at a point $v\in T^1M$ equals the unit sphere in 
$v_{\mathbb{K}}^\perp$, i.e. the 
sphere of all vectors
$w\in T^1M$ which are orthogonal to the $\mathbb{K}$-line
spanned by $v$. 
We construct a 
push-forward of the measure
$\mu$ to ${\cal S}$ as follows. 

Let $z\in {\rm supp}(\mu)$. 
Then $z=(x,y,F_1,F_2,F_3)$ where 
$x,y\in {\cal F\cal T}$ and where $F_i\in {\cal F}$.
The two framed tripods $x,y$ define two 
oriented ideal triangles contained in an immersed totally geodesic
hyperbolic plane in $M$. Let 
$T_x,T_y$ be the center triangles of 
these oriented ideal triangles. The side length of 
$T_x,T_y$ is $2r$. The vertices of $T_x,T_y$ 
depend smoothly on $x,y$, and the orientation
of $T_x,T_y$ 
determines a cyclic order of the vertices of $T_x,T_y$.

The order of the components of 
the point $z$ determines an order of the boundary components
of the skew-pants $\hat P(z)$. More precisely, the
tripods $x,y$ and the first two frames 
$F_1,F_2$ in the triple of frames from $z$ 
determine two geodesic arcs which connect
the two footpoints of the tripods, and 
the concatenation of these arcs is freely homotopic 
to a boundary geodesic
$\alpha$ of the 
skew-pants defined by $z$. 
As $\alpha$ is a boundary geodesic of the 
oriented pair of pants $\hat P(z)$, it is oriented.

Let $(u_x^1,u_x^2,u_x^3)$ be the ordered triple of vertices of 
the triangle $T_x$, and let $(u_y^1,u_y^2,u_y^3)$ be the ordered
triple of vertices of the triangle $T_y$.
The order of these vertices is chosen in such a way that
they define the orientation of $T_x,T_y$ and that moreover 
the oriented geodesic arc connecting $u_x^1$ to $u_x^2$
(or connecting $u_y^1$ to $u_y^2$)
crosses through the first connecting 
arc for the tripods 
in the triple (up to a homotopy which moves points at
most a distance $\kappa_4e^{-R}$ where $\kappa_4>0$ is as
in Section \ref{constructing}). Thus the geodesic
segment in the homotopy class determined by $\hat P(z)$
which connects $u_x^2$ to $u_y^1$ is contained in the
$\kappa_4e^{-R}$-neighborhood of the boundary geodesic $\alpha$
of $\hat P(z)$.

The points $u_x^1,u_y^2$ depend smoothly on $z$ and hence
the same holds true for 
the geodesic segment $\beta(z)$ 
which connects $u_x^1$ to $u_y^2$ and which is contained 
in the homotopy class 
determined by the skew-pants. The geodesic segment $\beta(z)$ 
is contained in the $\kappa_4e^{-R}$-neighborhood of a 
boundary geodesic of $\hat P(z)$ distinct from $\alpha$.

There is a unique geodesic arc $\eta$ in $M$ 
which connects the closed geodesic 
$\alpha$ and the geodesic arc 
$\beta$ and which is the shortest arc with this
property in the homotopy class relative to 
$\alpha,\beta$ determined by the skew-pants $\hat P(z)$.
The initial velocity $\eta^\prime$ of $\eta$ 
is a unit tangent vector with foot-point on $\alpha$ which 
is orthogonal to $\alpha^\prime$.
The angle between $\eta^\prime$ and a direction 
which is orthogonal to $\alpha^\prime\otimes \mathbb{K}$ 
is exponentially small in $R$. 

Write 
\[{\cal S}_\alpha={\cal S}\vert \alpha^\prime\]
and define
\[{\cal O}(z)\in {\cal S}_\alpha\]
to be the projection of $\eta^\prime$ into ${\cal S}_\alpha$.
Recall that this makes sense since $\alpha$ is oriented.
In particular, the footpoint of ${\cal O}(z)$ 
is the endpoint of the geodesic arc $\eta$ on $\alpha$.
This construction defines a map ${\cal O}:{\rm supp}(\mu)\to 
{\cal S}$ whose image is contained in the union of the
restriction of ${\cal S}$ to finitely many closed geodesics in $M$.
Let ${\cal S}_\mu=\cup_\alpha{\cal S}_\alpha$ 
be the union of  
these finitely many 
sphere bundles containing ${\cal O}({\rm supp}(\mu))$.

View ${\cal S}_\mu$ as 
a smooth (disconnected) manifold.
The restriction of the map ${\cal O}$ to the 
interior of ${\rm supp}(\mu)$ (which is a disconnected
smooth manifold as well) is smooth, moreover it is easily seen
to be open (as a map into ${\cal S}_\mu$). 
As a consequence, the 
restriction of the push-forward
${\cal O}_*(\mu)$ of $\mu$ to a component ${\cal S}_\alpha$
of ${\cal S}_\mu$ 
(where as before,
$\alpha$ is a closed geodesic in $M$) is contained
in the Lebesgue measure class. 
 
For each point 
$z=(x,y,F_1,F_2,F_3)\in {\rm supp}(\mu)$, the point 
${\cal O}(z)$ is uniquely determined by $x,y,F_1,F_2$.
Namely, the geodesics $\alpha$ and $\beta$ 
used for the construction of ${\cal O}$ only depend on these
data.

The involution $\iota$ of ${\cal F\cal T}^2$ 
which exchanges the tripods $x$ and $y$
and reverses the orders of the vectors in the tripods
(hence reversing
the orientation of $T_x,T_y$) preserves the Lebesgue measure.
Since the frame flow $\Psi^t$ preseves the Lebesgue measure,
it follows from the choice of the
functions $b_{x,R}$ and the definition of the 
measure $\mu$ that 
the involution on ${\cal F\cal T}^2\times {\cal F}^3$ 
which maps a point
$(x,y,F_1,F_2,F_3)$ to 
$(\iota(x,y),{\cal A}\Psi^R(F_3),{\cal A}\Psi^R(F_2),
{\cal A}\Psi^r(F_3))$ 
preserves $\mu$. 
Thus for every oriented closed geodesic $\alpha$ we have
\[{\cal O}_*(\mu)({\cal S}_{\alpha^{-1}})=
{\cal O}_*(\mu)({\cal S}_\alpha).\]

For a closed geodesic $\alpha$ in the support of 
${\cal O}(\mu)$ 
let 
\[\mu_\alpha={\cal O}_*(\mu)\vert {\cal S}_\alpha/
{\cal O}_*(\mu)({\cal S}_\alpha)\] 
be the normalization of 
${\cal O}_*(\mu)$ on ${\cal S}_\alpha$. 
Our next goal is to investigate the measures $\mu_\alpha$.
To this end
define for a closed oriented  
geodesic $\alpha$ in $M$ a fibre bundle map 
\[\rho_{{\alpha}}:
{\cal S}_{{\alpha}}\to 
{\cal S}_{{\alpha^{-1}}}\]
by requiring that $\rho_\alpha$ maps a point
in a fibre of ${\cal S}_\alpha$ to its negative, viewed as a point in a
fibre of ${\cal S}_{\alpha^{-1}}$.


\begin{lemma}\label{rotation}
For 
every oriented closed geodesic $\alpha$ in $M$,  
the measures $\mu_{\alpha^{-1}}$ and 
$(\rho_\alpha)_*\mu_\alpha$ are absolutely continuous, with
Radon Nikodym derivative in the interval
$[1-e^{-\kappa R}/\kappa,(1-e^{-\kappa R}/\kappa)^{-1}]$.
\end{lemma}
\begin{proof} By the definition of well connected
framed tripods, the following holds true.

Let $z=(x,y,F_1,F_2,F_3)\in {\rm supp}(\mu)$ 
and assume that ${\cal O}(z)\in {\cal S}_\alpha$.
The point $z$ determines a geodesic arc $\eta$ connecting
the footpoint of the tripod $x$ to the footpoint of $y$
which  is homotopic 
with fixed endpoints 
to a geodesic in the geometric
skew-pants $P(z)$ determined by $z$. 
The geodesic $\eta$ defines the good connection 
between the first two frames in the well connected tripods
$x,y$.

%
Choose a lift $\tilde \alpha$ of the geodesic $\alpha$ to 
$\tilde M$.
The tripods $x,y$ admit lifts $\tilde x,\tilde y$ 
to tripods in $\tilde M$ 
in such a way that a long side of each of 
the two totally geodesic hexagons
$H_R(\tilde x),H_R(\tilde y)\subset \tilde M$ 
is contained in the $\kappa_4e^{-R}$-neighborhood of 
$\tilde \alpha$. 
We also require that the 
footpoints of the tripods 
$\tilde x,\tilde y$ are connected
by a lift $\tilde \eta$ of the geodesic arc $\eta$.
These lifts then determine lifts 
$\tilde F_1,\tilde F_2$ 
of the frames $F_1,F_2$. They also determine
a lift $\tilde p$ of the footpoint $p$ of ${\cal O}(z)$.

Let $\sigma$ be the geodesic reflection about $\tilde p$.
Then $d\sigma(\tilde x),d\sigma(\tilde y)$ is a pair of tripods in 
$\tilde M$, and $d\sigma(\tilde F_1),d\sigma(\tilde F_2)$ are frames.
The projection $\eta(x,y,F_1,F_2)$ to $M$ of the quadruple
$(d\sigma(\tilde x),d\sigma(\tilde y),d\sigma(\tilde F_1),
d\sigma(\tilde F_2))$ determines the point 
$\rho_\alpha({\cal O}(z))$ on ${\cal S}_{\alpha^{-1}}$.
As a consequence, for every choice of a frame $\hat F_3$ so that
$\hat z=(\eta(x,y,F_1,F_2),\hat F_3)\in {\rm supp}(\mu)$ we have
${\cal O}(\hat z)=\rho_\alpha({\cal O}(z))$.

As the reflection $\sigma$ is an isometry and hence
it acts as a bundle automorphism on the bundle of 
framed tripods over $\tilde M$ 
and on the bundle of frames in $T\tilde M$ preserving the 
Lebesgue measure,
Fubini's theorem and Lemma \ref{mixing2} implies that  
the map $\rho_\alpha$ is absolutely continuous with respect to 
the measure $\mu_\alpha$ and
the measure $\mu_{\alpha^{-1}}$, 
with Radon Nikodym derivative contained in the interval
$[1-e^{-\kappa R}/\kappa,(1-e^{-\kappa R}/\kappa)^{-1}]$.
This shows the lemma.
\end{proof}

Let again $\alpha$ be a closed geodesic in $M$ and for $t\geq 0$ 
let $B_\alpha^t:{\cal S}_{{\alpha}}\to {\cal S}_{{\alpha}}$ 
be the
map induced by parallel transport of distance $t$.
The map $B_\alpha^t$ in turn 
is the projection of a map $B_{\tilde\alpha}^t$
which is defined as follows. Let $\tilde \alpha$ be a lift
of $\alpha$ to $\tilde M$, and let $B_{\tilde \alpha}^t$ be 
parallel transport of distance $t$ along $\tilde \alpha$.
Then $B_{\tilde \alpha}^t$ is the restriction
of a bundle automorphism of $T\tilde M$ defined by an isometry
of $\tilde M$ which preserves $\tilde \alpha$ and acts on
$\tilde \alpha$ as a translation. 
As in Lemma \ref{rotation}, we use this fact to conclude 
 
\begin{lemma}\label{rotation3}
The measures
$\mu_\alpha$ and $(B_\alpha^t)_*\mu_\alpha$ are absolutely continuous, with
Radon Nikodym derivative contained in the interval
\[[1-e^{-\kappa R}/\kappa,(1-e^{-\kappa R}/\kappa)^{-1}].\]
\end{lemma}

Recall that for a closed geodesic $\alpha$ in $M$
the monodromy of $\alpha$ is defined.
This monodromy is 
an isometry contained in the intropy group of
the tangent of $\alpha$ and hence it is an element  
$A\in SO(n-1)$ 
(or $A\in SU(n-1),A\in Sp(n-1),A\in {\rm Spin}(7)$).
For a given point $p\in \alpha$, it has 
a natural representative as an isometry 
of the $\mathbb{K}$-orthogonal complement of 
$\alpha^\prime$ in $T_pM$. 

The following observation is completely analogous to 
Lemma \ref{rotation} and Lemma \ref{rotation3}.
For its formulation, note that an isometry 
of the $\mathbb{K}$-orthogonal complement of 
$\alpha^\prime$ in $T_pM$ which commutes
with the monodromy of $\alpha$ determines
a bundle automorphism of ${\cal S}_\alpha$ commuting
with parallel transport. 

\begin{lemma}\label{rotation4}
Let $U\in SO(n-1)$ (or $U\in SU(n-1),
U\in Sp(n-1),U\in {\rm Spin}(7)$)
be an isometry of the $\mathbb{K}$-orthogonal
complement of $\alpha^\prime$ in $T_pM$ which 
commutes with the monodromy of $\alpha$. Then the 
measures $\mu_\alpha$ and $\mu_\alpha\circ U$ are absolutely continuous,
with Radon Nikodym derivative contained in the interval
$[1-e^{-\kappa R}/\kappa,(1-e^{-\kappa R}/\kappa)^{-1}]$.  
\end{lemma}

The bundle ${\cal S}_{{\alpha}}$ is a standard
sphere bundle over the circle. There is a natural 
Riemannian metric for this 
bundle which restricts to the round metric on each fibre. The 
length of the base equals the length 
$\ell(\alpha)$ of $\alpha$.
Let $d$ be the distance function on 
${\cal S}_{{\alpha}}$ induced by
this metric. The maps $\rho_\alpha:{\cal S}_{{\alpha}}\to
{\cal S}_{{\alpha^{-1}}}$ and $B_\alpha^t$ are isometries for these
metrics. Write $B_\alpha=B_{\alpha}^1$.

\begin{proposition}\label{matching}
If $G\not=SO(2m,1)$ for some $m\geq 1$ then 
there is a number $\theta>0$ not depending on $\alpha$, and 
there is a homeomorphism 
$\psi_\alpha:{\cal S}_{{\alpha}}\to 
{\cal S}_{{\alpha}}$ with 
\[(\rho_\alpha\circ B_\alpha\circ\psi_\alpha)_*\mu_\alpha=\mu_{\alpha^{-1}}\]
and $d(x,\psi_\alpha(x))\leq \theta e^{-\kappa R}$ 
for all $x\in {\cal S}_\alpha$.
\end{proposition}
\begin{proof} We observed before that the 
measures $\mu_\alpha$ are contained in the Lebesgue measure class.
Denote by $\omega$ the standard volume
form on the smooth oriented manifold 
${\cal S}_\alpha$; then 
we may assume that 
\[(\rho_\alpha\circ B_\alpha)^{-1}_*(\mu_{\alpha^{-1}})=g\omega,
\, \mu_\alpha= f\omega\] 
for continuous positive functions $f,g$ with
\[\int fd\omega=\int gd\omega=1.\]

Our goal is to show that there is a homeomorphism
$\psi_\alpha$ of ${\cal S}_\alpha$ which satisfies
$d(x,\psi_\alpha(x))\leq \theta e^{-\kappa R}$ for some 
$\theta >0$ and such that 
$\psi_\alpha^*(g\omega)=f\omega$.

Write $q=1-e^{-\kappa R}/\kappa$. 
By Lemma \ref{rotation3}, the function $f$ is invariant
under parallel transport up to a multiplicative 
factor of at most $q^{-1}$.
This implies the following.

Choose a parametrization of $\alpha$ by arc length on
the interval $[0,\ell]$.
Let $\pi:{\cal S}_\alpha\to \alpha$ be the natural projection, 
let $\omega_s$ be the standard volume form on 
$\pi^{-1}(s)$ and let 
$f_0:[0,\ell]\to (0,\infty)$ be the function
obtained by 
\[f_0(s)=\int_{\pi^{-1}(s)}f d\omega_s.\]
Then 
\[\int_a^b f_0 dt\in [q(b-a)/\ell,q^{-1}(b-a)/\ell]\]
for all $a<b$, and $\int_0^\ell f_0dt=1$.
Therefore if we define 
$\chi(t)=\ell\int_0^tf_0ds$ then $\chi:[0,\ell]\to[0,\ell]$ is a 
homeomorphism which moves points a distance 
at most $1-q$. Moreover,
the measure $\chi_*(\ell f_0dt)$ is the standard
Lebesgue measure $dt$ on the base $[0,\ell]$.

Lift the homeomorphism $\chi$
to a homeomorphism $\Psi:{\cal S}_\alpha\to {\cal S}_\alpha$
defined by 
\[\Psi(v)=\Vert_{\alpha[\pi(v),\chi(\pi(v))]}v.\]
Then the fibres of the bundle 
$\pi:{\cal S}_\alpha\to \alpha$ have volume one
for the volume form $\Psi_*(\ell f\omega)$. 
This implies that 
via moving fibres of ${\cal S}_{{\alpha}}$ 
with parallel transport and renormalization, 
it suffices to show the lemma 
under the additional assumption that each 
of the fibre integrals of 
$f$ and $g$ equals one.

Let $m={\rm dim}(M)-{\rm rk}(\mathbb{K})$
where ${\rm rk}(\mathbb{K})$ is the rank of 
$\mathbb{K}$ as an $\mathbb{R}$-vector space.
Let $A$ be the monodromy of $\alpha$. 
Then $A$ is an element of the orthogonal group
which fixes $\gamma^\prime$. 
If $G=SO(n,1)$ then there are no further constraints, and 
we have $A\in SO(n-1)=SO(m)$.
In the case $G=SU(n,1)$ the element $A$ also fixes
the image of $\gamma^\prime$ under the complex structure
and we have $A\in SU(n-1)<SO(2n-2)=SO(m)$.
Similarly, if $G=Sp(n,1)$ then $A$ fixes the 
quaternionic line spanned by $\gamma^\prime$
although perhaps not pointwise, and 
we can view $A$ as an element in 
$SO(4n-4)SO(4)<SO(m)SO(4)$.
Finally if $G=F_4^{-20}$ then 
$A$ fixes the Cayley line spanned by $\gamma^\prime$ and we can view
$A$ as an element in $SO(8)SO(8)=SO(m)SO(8)$.

We first consider the case that the component of the
identity 
$C(A)<SO(n-1)$ (or $C(A)<SU(n-1)<S(U(n)U(1)),
C(A)<Sp(n-1),C(A)<{\rm Spin}(9)$)
of the centralizer of $A$ in $SO(n-1)$ 
(or in $SU(n-1),Sp(n-1),{\rm Spin}(9)$) acts transitively on $S^{m-1}$,
viewed as 
a fibre of the bundle ${\cal S}_\alpha$. Observe that this holds
true if $M$ is a hyperbolic 3-manifold.

In this case
Lemma \ref{rotation4} shows that there is a function $\beta$
with values in the interval $[q,q^{-1}]$
such that $f\beta$ is a positive constant function.

Let $\omega_0$ be the 
smooth normalized volume form of the 
round metric on the round sphere $S^{m-1}$. 
For a number $\rho \in (0,1/4)$ 
consider for the moment an arbitrary continuous function
$h:S^{m-1}\to [1-\rho,1+\rho]$ with the property that
$\int (h-1)d\omega_0=0$. 
The function $h-1$ is bounded in norm by $\rho$.
Let $\Delta$ be the Laplacian of the round metric on $S^{m-1}$.
Then 
there is a unique function $\phi:S^{m-1}\to \mathbb{R}$ 
such that
\[\Delta(\phi)=h-1\text{ and }\int\phi\, d\omega_0=0.\]

Let $*$ be the Hodge star operator of the round metric
on $S^{m-1}$.
Schauder theory shows that the $(m-2)$-form
$\eta=* d*(\phi \omega_0)$ is bounded in norm by a constant multiple of 
$\rho$.

Let 
\[\nu_t=(1-t)\omega_0+th\omega_0.\]
Then for each $t$ the norm of the vector field $X_t$ defined by
\[\iota_{X_t}\nu_t=-\eta\]
is bounded from above by a constant multiple of $\rho$.
Let $\Lambda$ be the time-one map of the flow of 
the time dependent vector field $X_t$.
There is a number $\theta>0$ such that
$d(x,\Lambda x)\leq \theta\rho$
for all $x\in S^{m-1}$.
On the other hand, we have $\Lambda^*(h\omega_0)=\omega_0$.

We now apply this construction to 
the restrictions of the function $f$ to the 
fibres of ${\cal S}_\alpha\to \alpha$. These restrictions  
depend continuously on the fibre. As 
all functions and forms in the above construction
depend continuously on the function $h$ with respect to the 
$C^0$-topology, the fibrewise defined homeomorphisms which 
transform the volume form $\omega_s$ on the fibre
$\pi^{-1}(s)$  
to the volume form $f\omega_s$
determine a fibre preserving homeomorphism 
$\Lambda_f:{\cal S}_\alpha\to {\cal S}_\alpha$, and 
there is similarly a homeomorphism $\Lambda_g$.
Then 
\[\psi_\alpha=\Lambda_f^{-1}\circ\Lambda_g\]
(read from right to left) 
is a map with the properties stated in the proposition.
This concludes the proof of the proposition  
in the case that the 
component $C(A)$ of the identity of the 
centralizer of the 
monodromy $A$ of $\alpha$ acts transitively on the fibres of 
${\cal S}_\alpha\to \alpha$.

The general case is similar. By the
assumption $G\not=SO(2k,1)$, 
the dimension $m-1$ of the 
fibre of the sphere bundle ${\cal S}_\alpha$ is odd.
The group $C(A)$ can be described
as the group of all isometries of a fixed fibre of $S^{m-1}$ 
which preserve the generalized eigenspaces 
of the monodromy 
$A$ (and the complex structure for $G=SU(n,1)$ or the
quaternionic structure for $G=Sp(n,1)$). 
As $m-1$ 
is odd, $C(A)$ contains the element
$-{\rm Id}$.

For $v\in {\cal S}_\alpha$ 
the orbit $C(A)(v)$ 
of $v$ under the group $C(A)$ (which is  viewed  
as a group of isometries of the fibre
of ${\cal S }_\alpha$ containing $v$) is a
smooth submanifold
of $S^{m-1}$ which contains with $w$ the antipode $-w$. 
This submanifold is preserved by the monodromy
$A$ of $\alpha$. Thus if $v\in {\cal S}_\alpha$ is a 
vector 
with footpoint $\alpha(0)$ (for a parametrization 
of $\alpha$ by arc length as before) then 
\[{\cal C}=\cup_{t\in [0,\ell]}\Vert_{\alpha[0,t]}C(A)(v)\]
is a fibre bundle 
over $\alpha$ with smooth fibre which is invariant
under the antipodal bundle involution. 

Multiply the natural volume form $\omega_{\cal C}$ on 
${\cal C}$ with the functions $f,g$.
As ${\cal C}$ is invariant under the fibrewise antipodal map
and under parallel transport, 
the total integrals of $f$ and $g$ on ${\cal C}$ 
coincide. Since $C(A)$ acts transitively on the fibres
of ${\cal C}$, Lemma \ref{rotation4} shows that
the restrictions 
to ${\cal C}$ of the 
functions $f,g$ have values in an  
interval of the form 
$[c(1-e^{-\kappa R}/\kappa),c(1-e^{-\kappa R}/\kappa)^{-1}]$
for some $c>0$.
As a consequence, the argument for the case that
the action of $C(A)$ on the fibres of ${\cal S}_\alpha$ is 
transitive can be applied to the manifold ${\cal C}$
and yields a homeomorphism of ${\cal C}$ with the properties
required in the lemma (where we may have to adjust the constant
$\theta$ to take into account the various geometries of the
manifolds $C(A)(v)$).

Now the orbits of $C(A)$ form a compact family of manifolds, and the
functions $f,g$ are globally defined and continuous.
Thus 
carrying out this construction separately on each 
of the fibre bundles constructed from the orbits of 
$C(A)$ yields a 
homeomorphism $\psi_\alpha$ of ${\cal S}_\alpha$ 
as claimed.
\end{proof}



For $x\in {\rm supp}(\mu)$ let as before $P(x)$
be the geometric skew-pants defined by $x$.
We have (compare \cite{KM12} for the case 
$\tilde M={\bf H}^3)$

\begin{lemma}\label{wellattachedconclude}
For sufficiently small $\delta<\delta_4$ as in the definition of 
the measure $\mu$, the following holds true.
Let 
$x,y\in {\rm supp}(\mu)$; if ${\cal O}(x)\in {\cal S}_\alpha$ 
and if
${\cal O}(y)=(\rho_\alpha\circ B_\alpha\circ \psi_\alpha)(x)$ where
$\psi_\alpha$ is as in Proposition \ref{matching}, 
then $\hat P(y)$ is well attached to $\hat P(x)$ along $\alpha$. 
\end{lemma}
\begin{proof}
Let $\tilde\alpha$ be a lift of $\alpha$ to $\tilde M$.
Then $\alpha$ is the quotient of $\tilde \alpha$
by a loxodromic isometry $\Lambda\in G$. The translation length
of $\Lambda$ equals the length of $\alpha$. 
The rotational part of $\Lambda$ is the monodromy
$A\in SO(n-1)$ (or $A\in SU(n-1)$ or 
$A\in Sp(n-1)$) of $\alpha$. 

By Proposition \ref{boundaryskew}, for a number
$\epsilon <\pi/4$ depending on $\delta$,  
the monodromy $A$ of $\alpha$ is $\epsilon$-close to the identity.
In particular, 
there is a unique root $\Lambda^{1/2}$ of $\Lambda$
whose rotational part is $\epsilon$-close to the identity.
The map $\Lambda^{1/2}$ acts as an involution on the 
bundle ${\cal S}_\alpha$. 

Each skew-pants $P\in {\cal P}(R,\delta)$ which
contains $\alpha$ in its boundary has two seams
$\beta_1,\beta_2$ 
with endpoints on $\alpha$. 
Let $v_i$ be the unit tangent vector of $\beta_i$
on $\alpha$ $(i=1,2)$. We claim that
the distance in ${\cal S}_\alpha$ between 
$\Lambda^{1/2}(v_1)$ and 
$v_2$ 
is at most $\rho e^{-\zeta R}$ where 
$\rho>0,\zeta>0$ are universal constants.

To this end note that the seams $\beta_1,\beta_2,\beta_3$ 
of $P$ decompose $P$ into 
two right angled hexagons $H_1,H_2$ with geodesic sides in 
$M$. Equip $P$ with the structure of 
a geometric skew pants. For this geometric structure 
there are numbers $\delta_1<\delta,\delta_2<\delta,\delta_3<\delta$,
and there are approximate 
$\delta_1,\delta_2,\delta_3$-twisted ruled bands $B_1,B_2,B_3$ of size roughly
$6R-2\tau$ which are (locally) embedded in $P$
(see Section \ref{skewpants}). The bands $B_1,B_2,B_3$  
are separated by the seams $\beta_1,\beta_2,\beta_3$ of $P$ into
half-bands $B_i^j$ $(i=1,2,3,j=1,2)$.  
Up to an error which is exponentially small in $R$, 
the hexagon $H_j$ is composed of $B_1^j,B_2^j,B_3^j$ and 
a triangle which is exponentially 
close to an equilateral triangle of side length
$2r$  in a totally geodesic immersed hyperbolic 
plane in $M$.
The pairs of twisting angles of the half-bands $B_i^1,B_j^2$  
contained in $H_1,H_2$
coincide. But this means that the hexagons $H_1,H_2$ are
isometric up to an error which is exponentially small in $R$.

Now if $x\in {\rm supp}(\mu)$, if $P(x)=P$ 
and if ${\cal O}(x)\in {\cal S}_\alpha$ then 
up to exchanging $v_1$ and $v_2$, 
the vector ${\cal O}(x)$ is 
at distance at most a constant times
$e^{-\kappa R}/\kappa$ from $v_1$.
In particular, if $x,y$ are as in the lemma,
then the approximate twisted ruled bands of $P(x),P(y)$ 
which pass through the foot points of 
${\cal O}(x),{\cal O}(y)$ 
are well attached along a subarc of $\gamma$.

By the above, the second pair of approximate
twisted ruled band of $P(x),P(y)$ is well attached along
a subarc of $\gamma$ as well. Together this completes
the proof of the lemma.
\end{proof}

Each point $z\in {\rm supp}(\mu)$ 
defines a 
skew-pants $\hat P(z)\in {\cal P}(R,\delta)$.
The set ${\cal P}(R,\delta)$ is finite.
For $P\in {\cal P}(R,\delta)$ define
\[h(P)=\mu\{z\mid \hat P(z)=P\}.\]
Then $P\mapsto h(P)$ is a non-negative
weight function on the set ${\cal P}$ of all skew pants.
This weight function is invariant 
under the involution ${\cal J}$ of 
${\cal P}(R,\delta)$ which reverses the orientation.

For $P\in {\cal P}(R,\delta)$ 
let $\chi_P:{\rm supp}(\mu)\to [0,1]$ be the function defined by
$\chi_P(z)=1$ if the skew-pants $P(z)$ defined by $x$ equals $P$, and let
$\chi_P(z)=0$ otherwise.  Then we have 
\[h(P)=\int \chi_Pd\mu.\]
If $\gamma$ is a cuff of $P$ then 
the weighted measure $\chi_P\mu$ projects via the map ${\cal O}$
to a 
weighted measure $\chi_{P,\gamma}\mu_\gamma$ on ${\cal S}_\gamma$.
Since the measure $\mu$ is invariant under the map
which exchanges the two tripods in a point in 
${\cal F\cal T}^2\times {\cal F}^3$ and permutes the 
frames in ${\cal F}^3$, the total mass of the 
measure 
\[\chi_{P,\gamma}\mu_\gamma\] 
does not depend on the choice of the 
boundary geodesic $\gamma$ of $P$.

Let $\psi_\gamma:{\cal S}_\gamma\to {\cal S}_\gamma$ 
be as in Proposition \ref{matching}.
We may assume that 
$\psi_{\gamma^{-1}}=\psi_\gamma^{-1}$ where we identify
${\cal S}_{\gamma}$ with ${\cal S}_{\gamma^{-1}}$ 
with the obvious homeomorphism.

Define
\[h(P,P^\prime)=\int 
\chi_{P^\prime,\gamma^{-1}}(\rho_\gamma\circ B_\gamma\circ \psi_\gamma(x))
\chi_{P,\gamma}(x)d\mu_\gamma(x).\]
By construction and Lemma \ref{wellattachedconclude}, 
if $h(P,P^\prime)>0$ then the pants $P,P^\prime$ are 
well attached along $\gamma$. In particular, 
$P,P^\prime$ define an edge in the graph
${\cal G}(R,\delta)$.

The function $h(P,P^\prime)$ can be viewed as a   
non-negative weight
function on the edges of ${\cal G}(P,\delta)$.
Since $\psi_{\gamma^{-1}}=\psi_\gamma^{-1}$, 
by Proposition \ref{matching} 
this weight
function is symmetric: We have
\[h(P,P^\prime)=h(P^\prime,P)\]
for all $P,P^\prime$. Moreover, clearly
\[\sum_{P^\prime}h(P,P^\prime)=h(P)\]
which is equivalent to stating that this weight function is
admissible.

Theorem \ref{mainresult} is now a consequence
of Proposition \ref{solution} and Lemma \ref{solution1}.

\section{Concluding remarks}\label{even}

The proof of Proposition \ref{matching} is the only part of the
argument which 
is not valid for the groups $G=SO(2m,1)$ for $m\geq 2$
(with $SO(2,1)$ not relevant for the purpose of this work,
see however \cite{KM11}).

Namely, if $G=SO(2m,1)$ for some $m\geq 2$ then 
the monodromy of any closed geodesic 
$\alpha$ has a fixed unit vector.
If the eigenspace for the monodromy transformation
with respect to the eigenvalue one is 
bigger than one then the component of the identity of the 
subgroup of $SO(2m-1)$ of all elements
which commute with the monodromy transformation contains 
$-{\rm Id}$. In this case 
the argument in the proof of 
Proposition \ref{matching} is valid. 
However, if the dimension of 
the eigenspace 
for the eigenvalue one 
equals one then we can not use this argument.
Call such a periodic geodesic 
$\alpha$ with this property \emph{generic}.

For a generic closed geodesic 
$\alpha$, the bundle ${\cal S}_\alpha$
contains a sphere subbundle $\Sigma_\alpha$ 
whose fibre is a sphere of dimension $2m-3$. 
It is the sphere subbundle of the orthogonal complement of the 
one-dimensional eigenspace of the monodromy 
transformation for the eigenvalue one.
This sphere subbundle is invariant under parallel transport.

Choose a parametrization of $\alpha$ and invariant 
orientations of ${\cal S}_\alpha,\Sigma_\alpha$. 
For all $t$ and all $s\in [-\pi/2,\pi/2]$ 
the set $\Sigma_\alpha^s(t)$ of vectors in the
fibre ${\cal S}_\alpha(t)$ whose oriented 
distance to $\Sigma_\alpha(t)=\Sigma_\alpha^0(t)$ 
equals $s$ defines a
decomposition of ${\cal S}_\alpha(t)$ which is parametrized
on $[-\pi/2,\pi/2]$. For each $s$ the set 
$\cup_t\Sigma_\alpha^s(t)$ is 
invariant under parallel transport. 
In the glueing construction, we have to match a point 
in $\cup_t\Sigma^s_\alpha(t)$ with a point which is 
exponentially close to $\cup_t\Sigma^{-s}_\alpha(t)$.
Now if the measure of $\cup_t\cup_{s\leq 0}\Sigma^s_\alpha(t)$
is bigger than the measure of $\cup_t\cup_{s\geq 0}\Sigma^s_\alpha(t)$ 
and if the measure of a small neighborhood of 
$\cup_t\Sigma^0_\alpha(t)$ is exponentially small, then 
we can not match pants as required 
in the condition for incompressibility.

In spite of this difficulty, we believe that 
Theorem \ref{mainresult} holds true for even dimensional
hyperbolic manifolds. We also conjecture that it
is true for closed locally symmetric manifolds of the form
$M=\Gamma\backslash G/K$ where $G$ is a semisimple Lie group
with finite center, without compact factors and without factors
locally isomorphic to $SL(2,\mathbb{R})$, and where
$\Gamma<G$ is a cocompact irreducible lattice. 

The Kahn-Markovic argument does not seem to generalize
in an easy way to rank one locally symmetric manifolds
of finite volume. However, Theorem \ref{mainresult} is known
for non-compact finite volume hyperbolic 3-manifolds. 
We refer to \cite{BC14} for
a recent proof and references to related ealier results.

\bigskip
\noindent
MATHEMATISCHES INSTITUT DER UNIVERSIT\"AT BONN\\
ENDENICHER ALLEE 60\\
53115 BONN, GERMANY

\bigskip
\noindent
e-mail: ursula@math.uni-bonn.de


\begin{thebibliography}{BIW10}

\bibitem[BC14]{BC14} M.~Baker, D.~Cooper, 
{\em Finite-volume hyperbolic 3-manifolds contain
immersed quasi-Fuchsian surfaces}, arXiv:1408.4691.

\bibitem[BM00]{BM00} B. Bekka, M. Mayer, {\em Ergodic theory and topological
dynamics for group actions on  homogeneous spaces}, 
London Math. Soc. Lecture Notes, Cambridge Univ. Press 2000.

\bibitem[B12]{B12} N.~Bergeron,
{\em La conjecture des sous-groups de surfaces},
S\'eminaire Bourbaki 2011/2012, Expos\'es 1043--1058,
Asterisque No. 352 (2013), 429-458.

\bibitem[Bo63]{Bo63} A.~Borel,
{\em Compact Clifford-Klein forms of symmetric spaces},
Topology 2 (1963), 111--122. 

\bibitem[BH99]{BH99} M. Bridson, A. Haefliger, {\em Metric spaces of 
non-positive curvature}, Springer Grundlehren 319, Springer 1999.





\bibitem[B92]{B92} P.~Buser, {\em Geometry and spectra
of compact Riemann surfaces}, Birkh\"auser Progress in Math. 106,
Birkh\"auser Boston 1992.




\bibitem[CE75]{CE75} J.~Cheeger, D.~Ebin,
{\em Comparison theorems in Riemannian geometry},
North Holland/American Elsevier 1975.


\bibitem[CHH88]{CHH88} M.~Cowling, U.~Haagerup,
R.~Howe, {\em Almost $L^2$ matrix coefficients},
J. Reine Angew. Math. 387 (1988), 97--110.

\bibitem[Go99]{Go99} W.~Goldman,
{\em Complex hyperbolic geometry}, Oxford Math. Monographs,
Oxford Science Publications, Oxford 1999.


\bibitem[KM12]{KM12} J.~Kahn, V.~Markovic, {\em Immersing
almost geodesic surfaces in a closed hyperbolic three-manifold},
Ann. Math. 175 (2012), 1127--1190.


\bibitem[KM11]{KM11} J.~Kahn, V.~Markovic, {\em The good pants
homology and a proof of the Ehrenpreis conjecture},
arXiv:1101.1330.


\bibitem[LM13]{LM13} Y.~Liu, V.~Markovic,
{\em Homology of curves and surfaces in closed hyperbolic
$3$-manifolds}, arXiv:1309.7418.


\bibitem[Ra94]{Ra94} J. Ratcliffe,
{\em Foundations of hyperbolic manifolds}, 
Springer Graduate Texts in Math. 149, Springer New York 1994. 


\bibitem[S12]{S12} D.~Saric, {\em Complex Fenchel-Nielsen
coordinates with small imaginary parts},
arXiv:1204.5778.

\bibitem[W11]{W11} J.~Wolf, {\em Spaces of constant curvature},
Sixth edition. AMS Chelsea Publishing, Providence, RI 2011.

\end{thebibliography}
\end{document}